\documentclass[]{siamart1116}
\usepackage{amsfonts}
\usepackage{graphicx}
\usepackage{epstopdf}
\usepackage{algorithmic}
\usepackage{pgfplots}
\usepackage[caption=false]{subfig}
\ifpdf
  \DeclareGraphicsExtensions{.eps,.pdf,.png,.jpg}
\else
  \DeclareGraphicsExtensions{.eps}
\fi

\newcommand{\TheTitle}{Convergence rates for upwind schemes with rough coefficients}
\newcommand{\TheAuthors}{Andr\'e Schlichting and Christian Seis}

\headers{\TheTitle}{\TheAuthors}

\title{{\TheTitle}\thanks{Submitted to the editors \today.}}

\author{
  Andr\'e Schlichting\thanks{Institut f\"ur Angewandte Mathematik, Universit\"at Bonn (\email{schlichting@iam.uni-bonn.de}).}
  \and
  Christian Seis\thanks{Institut f\"ur Angewandte Mathematik, Universit\"at Bonn (\email{seis@iam.uni-bonn.de}).}
}

\usepackage{amsopn}

\ifpdf
\hypersetup{
  pdftitle={\TheTitle},
  pdfauthor={\TheAuthors}
}
\fi

\usepackage{amssymb}
\usepackage{graphicx}
\usepackage{stmaryrd}
\usepackage{enumerate}
\usepackage{mathtools}
\usepackage{relsize}

\def\Xint#1{\mathchoice%
{\XXint\displaystyle\textstyle{#1}}%
{\XXint\textstyle\scriptstyle{#1}}%
{\XXint\scriptstyle\scriptscriptstyle{#1}}%
{\XXint\scriptscriptstyle\scriptscriptstyle{#1}}%
\!\int}
\def\XXint#1#2#3{{\setbox0=\hbox{$#1{#2#3}{\int}$ }
\vcenter{\hbox{$#2#3$ }}\kern-.59\wd0}}
\def\avint{\Xint-}

\newcommand{\edge}{\mathsmaller{\mid}}
\newcommand{\grad}{\nabla}
\renewcommand{\div}{\grad\cdot}
\newcommand{\laplace}{\Delta}
\newcommand{\noto}{\,\,\not\!\!\longrightarrow}
\DeclarePairedDelimiter{\floor}{\lfloor}{\rfloor}

\newcommand{\N}{\mathbf{N}}
\newcommand{\F}{{\mathcal F}}

\newcommand{\R}{\mathbf{R}}

\newcommand{\Z}{\mathbf{Z}}

\DeclareMathOperator{\diam}{diam}
\def\opt{{\mathrm{opt}}}

\DeclareMathOperator{\spt}{supp}

\newcommand{\Ha}{\ensuremath{\mathcal{H}}}
\newcommand{\D}{\ensuremath{\mathcal{D}}}

\newcommand{\io}{\int_{\Omega}}
\def\per{{\mathrm{per}}}

\newcommand{\llb}{\llbracket}
\newcommand{\rrb}{\rrbracket}

\newcommand{\eps}{\varepsilon}

\newcommand{\tacka}{\, \cdot\,}

\DeclarePairedDelimiter{\abs}{\lvert}{\rvert}
\DeclarePairedDelimiter{\norm}{\lVert}{\rVert}
\DeclarePairedDelimiter{\bra}{(}{)}
\DeclarePairedDelimiter{\pra}{[}{]}
\DeclarePairedDelimiter{\set}{\{}{\}}
\DeclarePairedDelimiter{\skp}{\langle}{\rangle}

\newcommand{\cA}{\mathcal{A}}
\newcommand{\cB}{\mathcal{B}}
\newcommand{\cF}{\mathcal{F}}
\newcommand{\cT}{\mathcal{T}}
\newcommand{\EE}{\mathbb{E}}
\newcommand{\FF}{\mathbb{F}}
\newcommand{\PP}{\mathbb{P}}
\renewcommand{\L}{\mathcal{L}}

\begin{document}

\maketitle

\begin{abstract}
This paper is concerned with the numerical analysis of the explicit upwind finite volume scheme for numerically solving continuity equations. We are interested in the case where the advecting velocity field has spatial Sobolev regularity and initial data are merely integrable.
We estimate the error between approximate solutions constructed by the upwind scheme and distributional solutions of the continuous problem in a Kantorovich--Rubinstein distance, which was recently used for stability estimates for the continuity equation~\cite{Seis16a}. Restricted to Cartesian meshes, our estimate shows that the rate of weak convergence is at least of order~$1/2$ in the mesh size. The proof relies on a probabilistic interpretation of the upwind scheme~\cite{DelarueLagoutiere11}.
We complement the weak convergence result with an example that illustrates that for rough initial data no rates can be expected in strong norms. The same example suggests that the weak order $1/2$ rate is optimal.
\end{abstract}



\section{Introduction}

This paper is concerned with the numerical analysis of the explicit upwind finite volume scheme for solving linear conservative transport equations. We are interested in situations in which the coefficients in the equation are rough, but still within the range in which the associated Cauchy problem is well-posed. To be more specific, we consider nearly incompressible advecting velocity fields with spatial Sobolev regularity and configurations that are integrable but not necessarily bounded. This is the setting studied by DiPerna and Lions in their original paper \cite{DiPernaLions89}.

The goal of this work is an estimate of the error of the numerical scheme. In our main result, we show that the rate of convergence of the approximate solution given by the explicit upwind scheme towards the unique weak solution of the continuous problem is at least of order $1/2$ in the mesh size, uniformly in time. Our bound is valid for uniform \emph{Cartesian meshes}\footnote{All control elements are isometric axis-parallel rectangular boxes} only, but possible extensions to more general meshes are discussed. To measure the numerical error we use nonstandard distances from the theory of optimal mass transportation which appear to be natural in the context of continuity equations~\cite{Seis13b,Seis16a,Seis16b}. As these distances metrize weak convergence, the present paper provides a bound on the order of {\em weak} convergence. We will moreover see that, in general,  strong convergence rates cannot be expected for rough initial data. In this sense, the choice of weak convergence measures is optimal. Our computations moreover suggest that the order $1/2$ rate is sharp.

Considering coefficients under low regularity assumptions appears to be natural in the context of fluid dynamics, for instance, for problems described by compressible or incompressible inhomogeneous Navier--Stokes equations, or engineering questions related to fluid mixing, which attracted much interest recently \cite{Thiffeault12}. The present work can be considered as a first step towards the error analysis of numerical schemes approximating model problems that feature more general (also nonlinear) transport phenomena with rough coefficients.

Part of our analysis is built on a probabilistic interpretation of the upwind scheme similar to the one discussed by Delarue and Lagouti\`ere in \cite{DelarueLagoutiere11} (in the context of Lipschitz vector fields), and the canonical  representation of approximate solutions by the flow map induced by it. This interpretation recently guided the duo jointly with Vauchelet to new error estimates for the upwind scheme modeling transport with (one-sided) Lipschitz vector fields \cite{DelarueLagoutiereVauchelet16}. In a certain sense, the present work combines ideas from these two works with some novel stability estimates for continuity equations recently obtained in~\cite{Seis16a}. (See also \cite{Seis16b} for optimal estimates.)

\emph{Outline of the paper.}  In Section \ref{S:2}, we introduce the continuous model and the upwind finite volume scheme; we present and discuss our main result and illustrate it by numerical simulations.  Properties of the continuous model are collected in Section \ref{S:3}. Section \ref{S:5} contains a brief summary of tools from optimal mass transportation that are relevant in the analysis.
 In Section \ref{S:4}, we derive properties of the numerical scheme.
 Section \ref{S:6} is devoted to the proof of the error estimates. In Section \ref{Ex1} we propose an example which suggests the optimality of our main result. We conclude this paper with a discussion in Section \ref{S:discussion}.

\section{Setting and results}\label{S:2}

Since most our our error analysis is conducted for arbitrary (though regular) mesh geometries, we will in the following present the setting for general meshes. Whenever our argumentation is restricted to Cartesian meshes this will be emphasized. This strategy allows the reader to easily identify the obstacles that have to be overcome in order to generalize our result and is thus advantageous for future research. Again, we caution the reader that our main result is valid for \emph{Cartesian meshes} only.

\subsection{The continuous problem}
Let $\Omega$ be a bounded polyhedral domain in $\R^d$. The conservative transport of a quantity~$\rho$ with initial configuration $\rho_0$ by a vector field $u$ is modeled by the Cauchy problem for the continuity equation
\begin{equation}
\label{1}
\begin{cases}
\partial_t \rho + \div (u\rho) = 0& \mbox{in }[0,T]\times \Omega,\\
\quad\rho(0,\tacka) = \rho_0&\mbox{in }\Omega.
\end{cases}
\end{equation}
We are interested in vector fields with no flux across the boundary of the domain,
\begin{equation}
\label{12}
u\cdot \nu = 0\quad\mbox{on }\partial\Omega,
\end{equation}
where $\nu$ denotes, as usual, the outward normal on $\partial \Omega$. The set of equations is chosen in such a way that the total ``mass'' $\int \rho\, dx$ is (formally) conserved.

It is well known that in the case of smooth vector fields, solutions can be constructed via the method of characteristics. Out of the smooth setting, the analytical treatment of the equation was initiated by DiPerna and Lions \cite{DiPernaLions89}, who developed the theory of so-called  \emph{renormalized solutions}. The authors derive uniqueness and stability properties of renormalized solutions and show that distributional solutions are renormalized if the advecting velocity field satisfies certain regularity assumptions. These are $u\in L^1((0,T);W^{1,p}(\Omega))$ with $(\div u)^-\in L^1((0,T);L^{\infty}(\Omega))$, where we have used the superscripted minus sign to denote the negative part of a number. In the following, a solution of \eqref{1} will always be the unique distributional solution constructed in \cite{DiPernaLions89}.
The DiPerna--Lions theory was later extended to vector fields with bounded variation regularity $u \in L^{1}((0,T);BV(\Omega))$ by Ambrosio \cite{Ambrosio04}.

Stability estimates for the continuity equation are very recent. In \cite{Seis16a}, the second author chose an optimal transportation approach that yields quantitative estimates for the distance of two solutions corresponding to nearby velocity fields and nearby initial configurations. The optimality of this approach is discussed in \cite{Seis16b}. The latter works mirror analogous results in the Lagrangian framework derived earlier by Crippa and De~Lellis in \cite{CrippaDeLellis08}. In these quantitative estimates, however, the case $p=1$ (and also $BV$) is excluded.

The present work builds up on \cite{Seis16a}: We study the distance between approximate solutions constructed by the explicit finite volume upwind scheme---a numerical approximation of \eqref{1}---and the unique weak solution to the original problem. An upper bound on this distance thus serves as an estimate for the numerical error. As in \cite{CrippaDeLellis08} and \cite{Seis16a}, we need to restrict to the case $p>1$.

Because our numerical scheme is \emph{explicit}, a stability condition has to be implemented, which requires to control the velocity field \emph{uniformly in space}. We thus impose in addition that $u\in L^1((0,T);L^{\infty}(\Omega))$.

\subsection{The numerical scheme}

The upwind finite volume scheme is the most classical, stable, and monotone numerical approximation of the continuity equation \eqref{1} (see, e.g., \cite{EymardGallouetHerbin00,LeVeque02}). It is formulated on a tessellation of the physical space $\Omega$ and describes the evolution of cell averages
by means of the flux over the cell boundaries.

Even though our result will be valid for Cartesian meshes only, we start with the description of the upwind scheme for quite general mesh geometries. In fact, most parts of our analysis hold true for general meshes. For this reason, we will work under weaker assumptions in the majority of the paper and restrict to the Cartesian setting only where needed. Our hope is to remove this restriction in some future work.

We consider a tessellation $\cT$ of the domain $\Omega$, that is, $\cT$ is a family of closed, connected polyhedral subsets (called control volumes or, simply, cells) of $\R^d$ with disjoint interiors and such that $\overline{\Omega} = \bigcup_{K\in\cT} K$.
The surface of each control volume $K\in\cT$ consists of finitely many flat, closed and connected $d-1$ dimensional faces. If $K$ and $L$ are two neighboring control volumes, we write $K\sim L$. In this case, we use the notation $K\edge L$ to denote the joint edge of $K$ and $L$ and we use $\abs{K\edge L}$ to denote its $(d-1)$-dimensional Hausdorff area $\Ha^{d-1}(K\edge L)$. Moreover, the outward normal on $K$ at the edge $K\edge L$ will be denoted by $\nu_{KL}$, so that $\nu_{KL} = -\nu_{LK}$.
On any edge $K\edge L$, we define the relative inverse length scale $\tau_{KL} = \frac{|K\edge L|}{|K|}$, where by abuse of notation $\abs{K}$ is the usual $d$-dimensional volume of the control element $K$. Finally, the mesh size $h$ is the maximal diameter of the volumes
\begin{equation*}
h = \max_{K\in \cT} \diam K.
\end{equation*}
It is necessary to ensure a certain uniform regularity of the mesh, that essentially guarantees that the control volumes do not degenerate as $h\to 0$. On the level of the numerical analysis presented below, this regularity assumption must imply that geometrical constants in estimates remain bounded as $h\to 0$.
More precisely, we assume that there exists a constant $C>0$ such that the trace estimate
\begin{equation}
\label{e:trace}
\|f\|_{L^1(\partial K)} \le C \left( \|\grad f\|_{L^1(K)} + h^{-1} \|f\|_{L^1(K)}\right)
\end{equation}
holds true uniformly as $h\to0$, for any function $f$ and any cell $K\in\cT$. The proof of the trace estimate is fairly standard and can be found, for instance, in \cite[Chapter 4.3]{EvansGariepy92}. An immediate consequence of this estimate with $f\equiv 1$ is the uniform isoperimetric property of the control elements
\begin{equation}\label{e:isoper}
  \max_{K\in \cT} \frac{\abs{\partial K}}{\abs{K}} \leq \frac{C}{h} .
\end{equation}

For a fixed time step size $\delta t$ specified in \eqref{e:CFL_num} and \eqref{CFL_true} below, we choose $N\in\N$ such that $N\delta t\le T$ and we set $t^n = n\delta t$ for any $n\in \llbracket 0,N\rrbracket$. Here and in the following, we use the notation $\llb a,b\rrb = [a,b]\cap \N_0$, where is $\N_0=\set{0,1,2,\dots}$ are the nonnegative integers.

We discretize the initial datum $\rho_0$ by assigning to each control volume $K\in\cT$ the mean of $\rho_0$ over that volume, i.e.,
\begin{equation}\label{e:disc_initial}
\rho_K^0 = \avint_K \rho_0\, dx.
\end{equation}
The upwind scheme takes into account only the flow over the edges. For each $L\sim K$, we consider the net outflow over $K\edge L$ per time interval $\left[t^n,t^{n+1}\right]$,
\begin{equation}\label{e:def:uKsigma}
u_{KL}^n  = \avint_{t^n}^{t^{n+1}} \avint_{K\edge L} u\cdot \nu_{KL} \,d\Ha^{d-1}\,dt.
\end{equation}
Following the sign conventions for the outward normals, we will sometimes use the antisymmetric relation $u_{KL}^n = - u_{LK}^n$.
We will furthermore distinguish between the flows inwards and outwards the control volume. Hence, for $L\sim K$, we write $u_{KL}^{n+} = (u_{KL}^n)^+$ and $u_{KL}^{n-} = (u_{KL}^n)^-$, where $(q)^+=\max\set{0,q}$ and $(q)^-=-\min\set{0,q}$ denote the positive and the negative part of the quantity $q\in \R$, respectively.

With these preparations, we are now in the position to introduce the explicit upwind finite volume scheme. For $n\in\llb 0,N\rrb$ and $K\in\cT$, we define iteratively
\begin{equation}
\label{e:upwind}
\rho_K^{n+1} = \rho_K^n -\delta t \sum_{L\sim K} \tau_{KL} \left(u_{KL}^{n+} \rho_K^n - u_{KL}^{n-}\rho^n_L\right) .
\end{equation}
Let us stress out, that due to the no flux condition~\eqref{12}, there are no boundary terms present in~\eqref{e:upwind}.
It will be beneficial for our probabilistic interpretation to rewrite the upwind scheme using the identities $u_{KL}^{n-}=u_{LK}^{n+}$ and $\abs{K} \tau_{KL} = \abs{L}\tau_{LK}$ as
\begin{equation}\label{e:upwind2}
\abs{K}\rho_K^{n+1}=  \sum_{L \in \cT}  \abs{L} \rho_L^n p_{LK}^n,
\end{equation}
where $p^n_{KL}$, defined by
\begin{equation}\label{e:def:p_KL}
 p_{KL}^n := \begin{cases}
    1- \delta t \sum_{L\sim K} \tau_{KL} u_{KL}^{n+} & \text{ if $L=K$,} \\
    \delta t \, \tau_{KL}\, u_{KL}^{n+} & \text{ if $L\sim K$,} \\
     0 &\text{ else,}
  \end{cases}
\end{equation}
will play the role of transition probabilities. It is a well-known fact that the upwind scheme is \emph{stable}, if $\delta t$ is chosen according to the Courant--Friedrichs--Lewy (CFL) condition
\begin{equation}\label{e:CFL_num}
  \forall K \in \cT : \qquad \sum_{L\sim K} p_{KL}^n = \delta t \sum_{L\sim K} \tau_{KL} u_{KL}^{n+} \leq 1.
\end{equation}
We will recall the proof of  stability in Lemma \ref{lem:prop_scheme} below. For our analysis, it will be convenient to impose a slightly stronger condition, which does not depend on the mesh. Therefore, we demand for some finite constant $C$
\begin{equation}
\label{CFL_true}
\max_{n\in\llbracket0,N\rrbracket} \delta t \, u_\infty^n \le C h \qquad\text{where}\qquad u^n_\infty := \avint_{t^n}^{t^{n+1}} \|u\|_{L^{\infty}}\, dt,
\end{equation}
which in a certain sense is sufficient for \eqref{e:CFL_num} in the case of Sobolev vector fields (cf.\ Lemma~\ref{lem:CFL}). The latter  condition in particular implies that in every time step the maximal length of each path line is at most of order $h$ (see, e.g., the proof of Lemma~\ref{L1}).

The {\em approximate solution} $\rho_h$ is defined in such a way that
\begin{equation}\label{e:rho_h}
\rho_h(t,x) = \rho^n_h(x)  = \rho^n_K\qquad\mbox{for almost every }(t,x)\in[t^n,t^{n+1})\times K.
\end{equation}
We will also write $\rho_h^0 = \rho_h(0,\tacka)$ for the approximate initial datum.

\subsection{Main results}

In the following,  we describe and interpret our main result. As announced in the previous subsection, our results are valid for Cartesian meshes only. It is thus necessary to restrict the admissible geometries for the domain $\Omega$. We call $\Omega$ \emph{compatible} to Cartesian tessellations, if it is a finite disjoint union of isometric axis-parallel rectangular boxes. In this case, $\Omega$ can be covered by control volumes $K\in \cT$, which are of the form $K = [a_1, a_1+ h_1]\times \dots\times [a_d,a_d + h_d]$ with edge lengths $h_i$ satisfying $h \le C h_i$ uniformly in $h$. The latter condition is equivalent to~\eqref{e:trace} via~\eqref{e:isoper}.

We are now in the position to state our main result.
\begin{theorem}\label{T1}
Suppose that $\Omega$ is a bounded domain in $\R^d$ that is compatible to Cartesian tessellations.

Let $p\in(1,\infty]$ and $q\in[1,\infty)$ be dual H\"older exponents, i.e., $1/p+1/q=1$. Let $u:[0,T]\times \Omega\to \R^d$ be such that
\begin{equation}
\label{e:ass:u}
\begin{cases} u\in L^1((0,T);L^{\infty}(\Omega)),\\
 \grad u\in L^1((0,T);L^p(\Omega)), \\
(\div u)^-\in L^1((0,T);L^{\infty}(\Omega)), \\
u \cdot \nu = 0  \text{ on } \partial\Omega,
\end{cases}\end{equation}
and let $\rho_0:\Omega\to \R$ be such that
\[
\rho_0 \in L^q(\Omega).
\]

Let $\rho, \rho_h: [0,T]\times\Omega\to \R$ denote, respectively, the solution to the continuity equation~\eqref{1} and the approximate solution defined by the explicit upwind finite volume scheme \eqref{e:disc_initial}, \eqref{e:def:uKsigma}, \eqref{e:upwind},  \eqref{e:rho_h} on a Cartesian mesh with mesh size $h$. Suppose that the CFL conditions \eqref{e:CFL_num} and \eqref{CFL_true} are satisfied.

Then there exists a constant $C$ such that for any $r>0$ and any $t\in [0,T]$, it holds
\begin{equation}\label{e:result}
\begin{aligned}
\MoveEqLeft \inf_{\pi\in\Pi(\rho(t,\cdot), \rho_{h}(t,\cdot))} \iint \log\left(\frac{|x-y|}{r} +1\right)\, d\pi(x,y)\\
& \le C\left(1+\frac{\sqrt{h \|u\|_{L^1(L^\infty)}} + h}{r}\right)\bra*{ \| \rho_0\|_{L^1} +  \Lambda^\frac1p \, \|\rho_0\|_{L^q}\, \| u\|_{L^1(W^{1,p})}},
\end{aligned}
\end{equation}
where $\Lambda = \exp\left(\|(\div u)^-\|_{L^1(L^{\infty})}\right)$.
\end{theorem}

Here we have dropped the dependency of norms on the domains for notational convenience. For instance, $L^1(L^p) = L^1((0,T);L^p(\Omega))$. We will stick to this convention in the following. In the case where $\Omega$ is convex, the term $\|u\|_{L^1(W^{1,p})}$ can be replaced by $\|\grad u\|_{L^1(L^p)}$, which is consistent with regard to dimensions.

The statement involves the notion of a Kantorovich--Rubinstein distance. The infimum on the left-hand side of \eqref{e:result} is taken over so-called transport plans $\pi$ that are joint measures on the product space $\Omega\times \Omega$ with marginals $(\rho-\rho_h)^+$ and $(\rho-\rho_h)^-$. Roughly speaking, the quantity on the left-hand side measures the minimal total cost that is necessary to transfer the configuration $\rho$ into the configuration $\rho_h$, if the transport over a distance $z$ costs $\log\left(z/r+1\right)$. Finding and characterizing the ``best'' transport plan $\pi$ is a central question in the theory of optimal transportation. The minimal total cost is a mathematical {\em distance} between $\rho$ and $\rho_h$ and \emph{metrizes weak convergence}. In Section~\ref{S:5} below, we will review properties of this (Kantorovich--Rubinstein) distance function that are relevant for the comprehension of this paper.

With the understanding that Kantorovich--Rubinstein distances metrize weak convergence, the statement in Theorem \ref{T1} can be seen as an {\em estimate on the order of weak convergence} of approximate solutions defined by the upwind finite volume scheme towards the unique solution to the continuous problem. For finite time intervals and small mesh sizes, the $\sqrt{h}$ term dominates the convergence rate as $ \sqrt{h \|u\|_{L^1(L^\infty)}} \gg h$. For $r = \sqrt{h} \sim  \sqrt{h \|u\|_{L^1(L^\infty)}}$, the right-hand side becomes independent of $h\ll 1$ and the estimate turns into
\[
 \inf_{\pi\in\Pi(\rho(t,\cdot), \rho_{h}(t,\cdot))} \iint \log\left(\frac{|x-y|}{\sqrt{h}} +1\right)\, d\pi(x,y) \le C,
 \]
uniformly in $h\ll 1$ and $t\le T$ for some fixed finite $T$. In words: Under the assumptions of the theorem, \emph{the order of weak convergence of approximate solutions towards the solution of the continuous problem is at most $1/2$}.

We will see in Section \ref{Ex1} below that our result is optimal in two regards: On the one hand, we will show that for general rough initial data one cannot expect convergence rates in strong norms. More precisely,  for any $s\in(0,1)$ we find  an initial configuration in $L^1$ such that
\[
h^{s-1}\|\rho-\rho_h\|_{L^1(L^1)} \noto 0\quad
\] 
as $h\to 0$. On the other hand, we can show for the same example that the rate of weak convergence is at least of order $h^{1-s/2}$.
This coincides with our findings in the limiting case where $s\nearrow 1$.

Order $1/2$ convergence for the upwind scheme has been known for a long time in the case of regular (e.g., Lipschitz) vector fields: Even though the scheme is formally order~$1$, for nonsmooth initial configurations, the optimal convergence rate falls down to $1/2$. Among the many papers proving this result, we mention \cite{Kuznecov76,Peterson91,VilaVilledieu03,Despres04,
MerletVovelle07,Merlet07,CockburnDongGuzmanQian10,
DelarueLagoutiere11,DelarueLagoutiereVauchelet16,AB16}. To the best of our knowledge, in this paper, we provide the first analytical results on the convergence rates in the case of vector fields with spatial Sobolev regularity. Numerical evidence for this rate was reported earlier by Boyer~\cite{Boyer12}, see also Subsection \ref{S:numerics} below. Convergence (without rates) of the scheme in the DiPerna--Lions setting was obtained by Walkington \cite{Walkington05} and Boyer \cite{Boyer12}.

The reason for the loss in the convergence rate for nonsmooth initial data is the occurrence of {\em numerical diffusion}. In a certain sense, the approximate scheme behaves like the diffusive approximation
\[
\partial_t \rho + \div\left(u\rho\right)  - h\laplace \rho = 0,
\]
which on the level of the Lagrangian variables is understood as a stochastic differential equation
\begin{equation}\label{intro:SDE}
d \psi_t  = u(t,\psi_t)\, dt + \sqrt{2h}\, dW_t ,
\end{equation}
where $W_t$ is a Brownian motion in $\R^n$. This motivates that the upwind scheme has a probabilistic interpretation.
Recently, Delarue, Lagouti\`ere and Vauchelet~\cite{DelarueLagoutiereVauchelet16} interpreted the upwind scheme in the form~\eqref{e:upwind2} as Kolmogorov forward equation of a Markov chain on the mesh $\cT$ with jump probabilities given by $p_{KL}^n$~\eqref{e:def:p_KL}. By doing so, they were able to prove an $h^{1/2}$-rate of convergence for the upwind scheme applied to the continuity equation with a one-sided Lipschitz vector field. In our case of Sobolev vector fields, we define $\psi_t$ as a continuous state Markov chain on $\Omega$ with a suitable jump kernel between the elements of the mesh. This Markov chain is a time-discretized version of the stochastic differential equation~\eqref{intro:SDE} with a noise term still depending on the details of the mesh (cf.\ Lemma \ref{lem:ave:Psi} and \eqref{e:Psi:difference} below). Moreover, this noise term determines the $h^{1/2}$-rate of convergence (cf.\ Lemma \ref{L8} below).

At the end of this subsection, we try to convince the reader  that estimates on  logarithmic distances as the one in our main result appear quite naturally in the context of continuity and transport equations. We will do so on the Lagrangian level, that is, we consider the flows (cf.\ \eqref{e:flow:sol} below) $\phi$ and $\phi_h$ of two bounded Lipschitz vector fields $u$ and $u_h$. It is not difficult to see that
\begin{equation}
\label{e:log_estimates}
\log\left(\frac{|\phi(t,x) - \phi_h(t,x)|}r +1\right) \le \|\grad u\|_{L^1(L^{\infty})} + \frac{1}r\|u-u_h\|_{L^1(L^{\infty})},
\end{equation}
uniformly in $t$ and $x$. Hence, choosing $r= r_h = \|u-u_h\|_{L^1(L^{\infty})}$, we see that the velocity gradient controls the logarithmic distance of trajectories relative to the distance of the vector fields~$r_h$. The argument for this estimate is straight-forward and follows from the calculation
\[
\left|\frac{d}{dt} \log\left(\frac{|\phi(t,x) - \phi_h(t,x)|}r +1\right)\right| \le \frac{|u(t,\phi(t,x)) - u_h(\phi_h(t,x))|}{|\phi(t,x) - \phi_h(t,x)| + r}
\]
by triangle inequality and integration. A generalization of \eqref{e:log_estimates} to the case where $u$ has spatial Sobolev regularity is due to Crippa and De~Lellis \cite{CrippaDeLellis08}. In the Eulerian framework, an analogous estimate was derived recently in \cite{Seis16a,Seis16b}.

\subsection{Numerical experiments} \label{S:numerics}

We now present some numerical findings in favor of our analytical results. In two series of experiments, we applied the explicit upwind finite volume scheme to initial data with jump discontinuities. In the first series, the scheme is run with a constant vector field, in the second one, we used a stationary H\"older regular field belonging to $W^{1,p}$ for any $p$ with $1\le p<2$. In both cases we find that the convergence rate is at least of order $1/2$ if the error is measured in the $L^1$ as well as the $H^{-1}$ norm.

\begin{figure}[!b]\label{fig1}
\centering
\subfloat[Exact solution ]{\includegraphics[scale=.07]{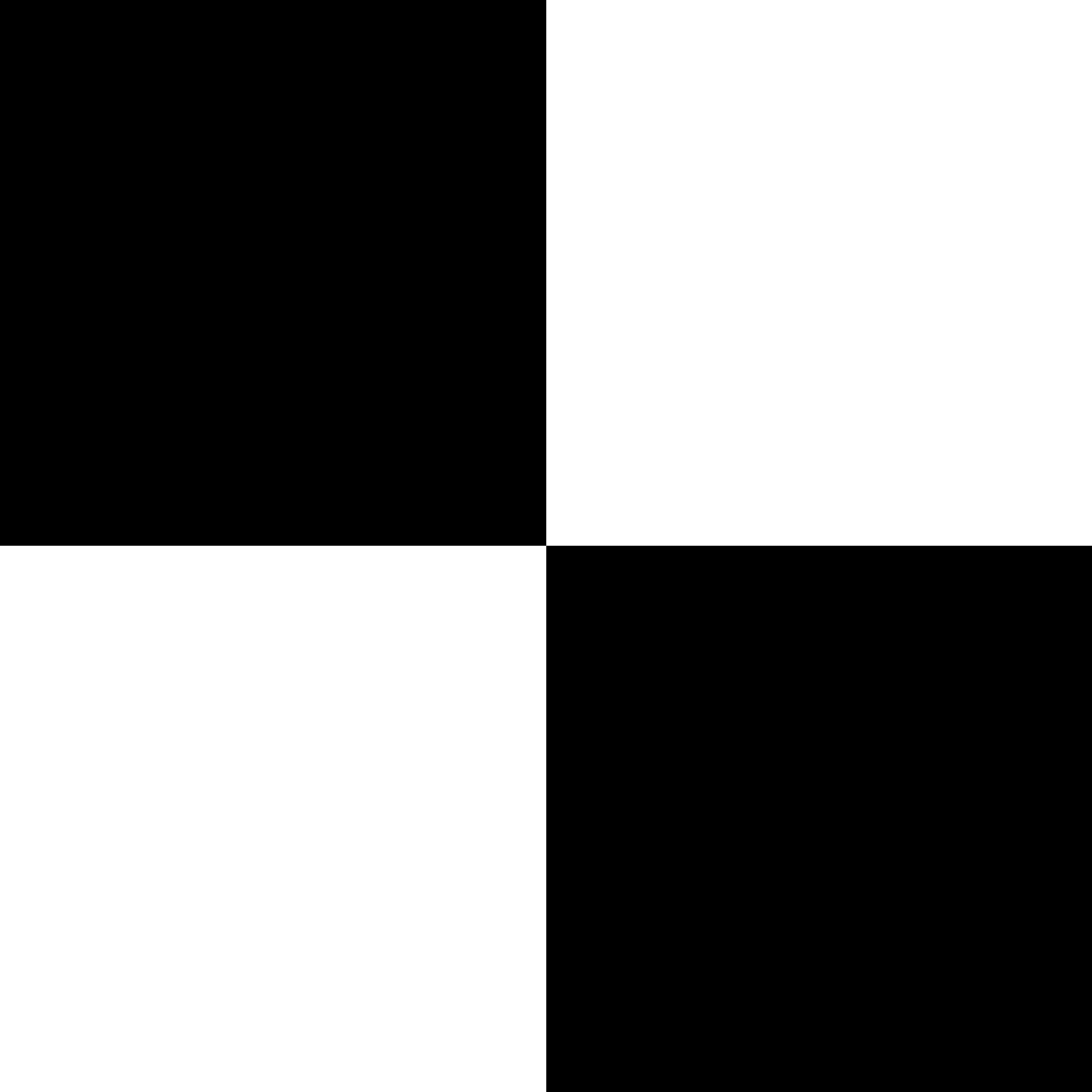}}
\hspace{2em}
\subfloat[Constant vector field $u_c$]{\includegraphics[scale=.07]{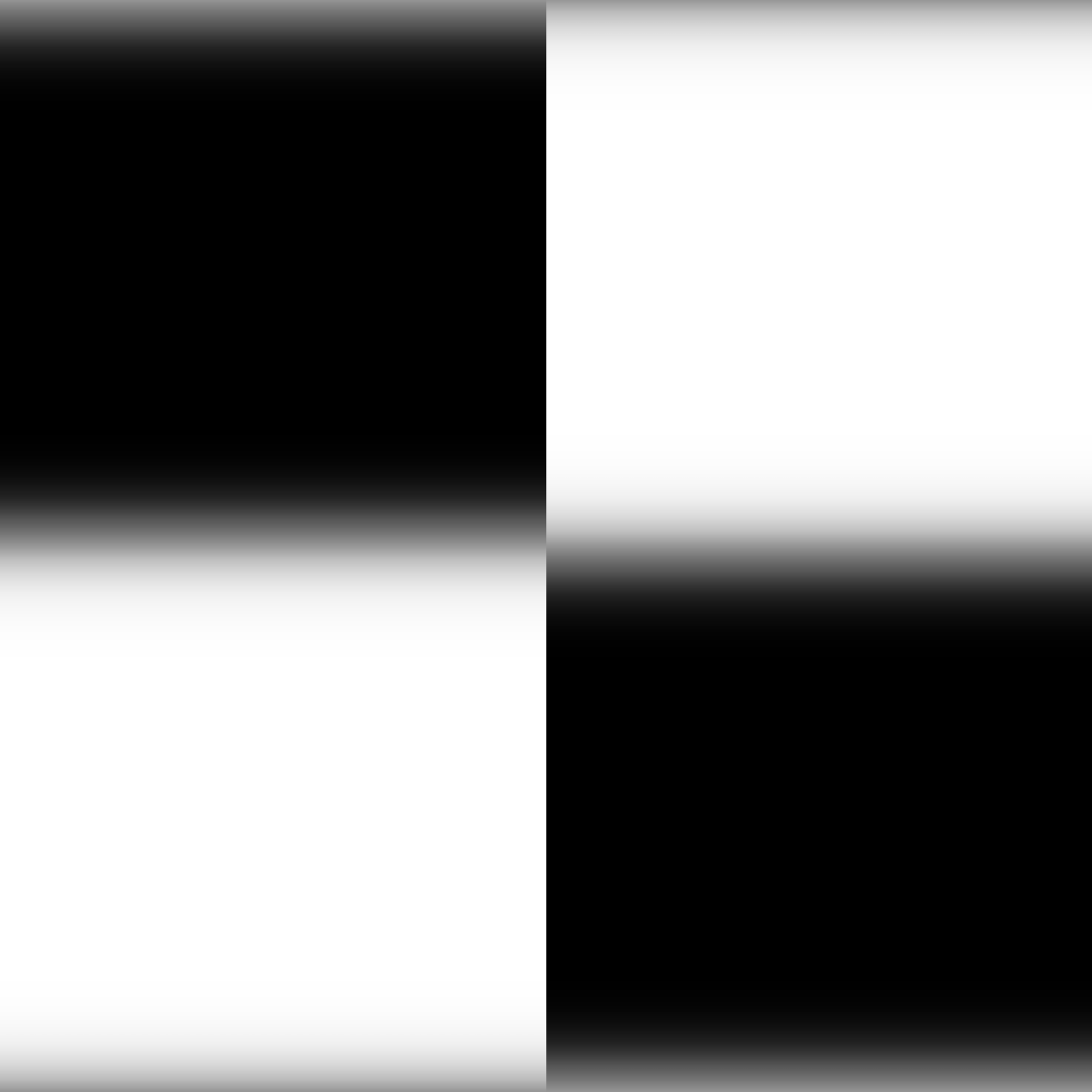}}
\hspace{2em}
\subfloat[Sobolev vector field $u_S$]{\includegraphics[scale=.07]{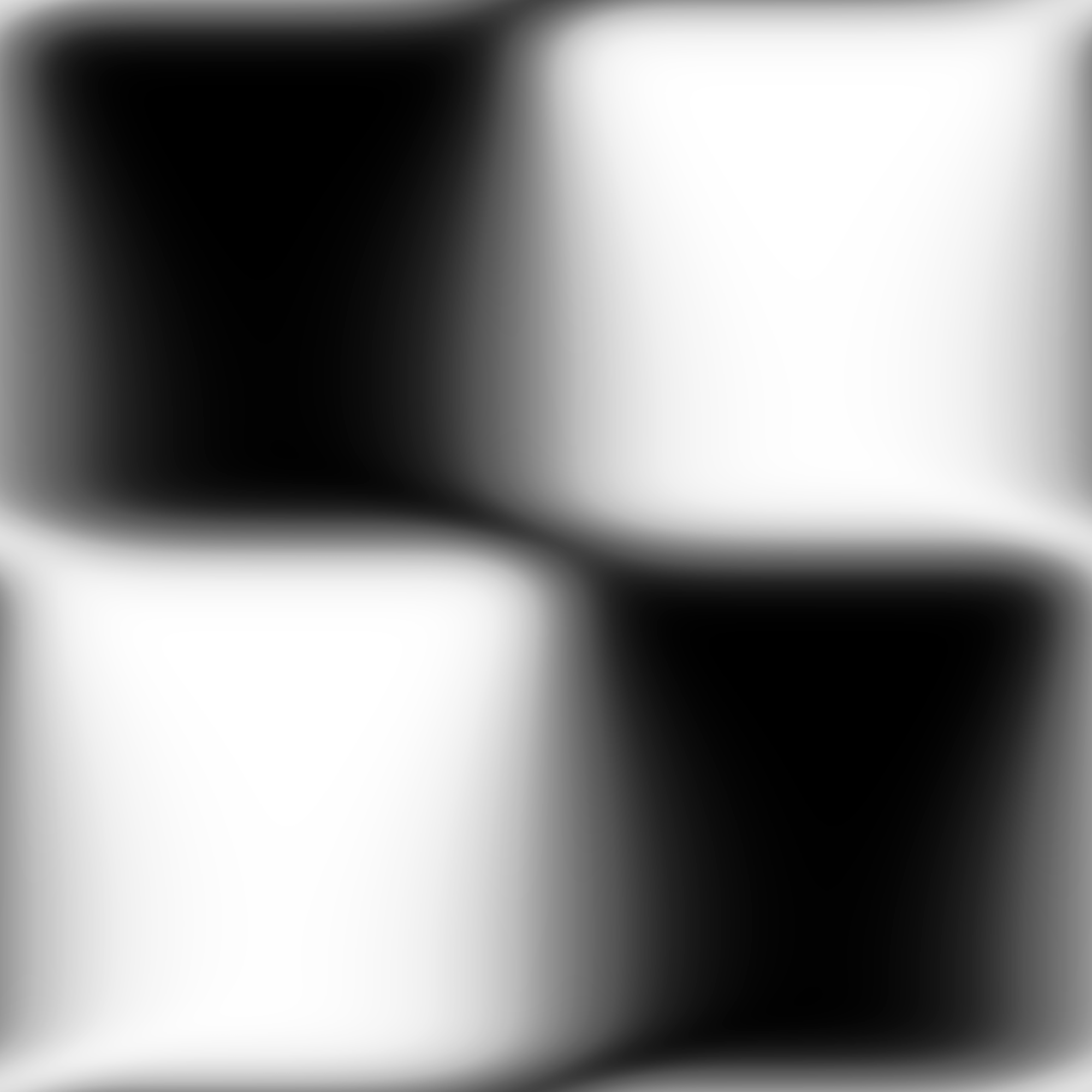}}
\caption{Illustration of numerical diffusion of upwind scheme (Mesh size $h = 2^{-10}$).}
\end{figure}

Let us describe our experiments in more detail. We consider the continuity equation on the two-dimensional unit torus $\Omega \cong [0,1)_{\per}^2$. As initial datum, we consider
\[
\rho_0(x)  = \begin{cases} 1 &\mbox{if }x\in \left[0,\frac12\right)^2 \cup \left[\frac12,1\right)^2,\\
-1 & \mbox{else.}\end{cases}
\]
In our first series of experiments, we run the experiment with the constant vector field $u_c = (0,1)^T$. In the second experiment, we choose $u_S = (v,.5)^T$ with $v=v(x_2)$ given by
\[
v(x_2) = \begin{cases} \sin^{\frac12}\left(2\pi x_2\right) &\mbox{if }x_2\in \left[0,\frac12\right),\\
-\sin^{\frac12}\left(-2\pi x_2\right) &\mbox{if }x_2\in \left[\frac12,1\right).
\end{cases}
\]
It is clear that $v$ is H\"older continuous with exponent $1/2$ and  belongs to $W^{1,p}$ for any $1\le p <2$.
In both cases, we flip the sign of the vector field at time $T=1$. As the continuity equation is time reversible, the exact solution reaches the initial state at time $T=2$. Notice that both vector fields are divergence-free with $\|u\|_{L^{\infty}}=1$.

We run the simulations on a Cartesian mesh of size $h$ ranging from $2^{-11}$ to $2^{-5}$. The time step size is fixed to $\delta  t = h/4$.

Figure \ref{fig1} illustrates the effect of the numerical diffusion in the upwind scheme in both experimental series.
The reference is the exact solution displayed in the plot on the left. The plot in the middle is computed with the upwind scheme using $u_c$, the plot on the right is computed using $u_S$. As it is clear from the definition of the scheme, diffusion can only happen in the direction of the flow. For this reason, the vertical transitions at $x_1=0$ and $x_1=1/2$ remain sharp under the constant vector field $u_c$.

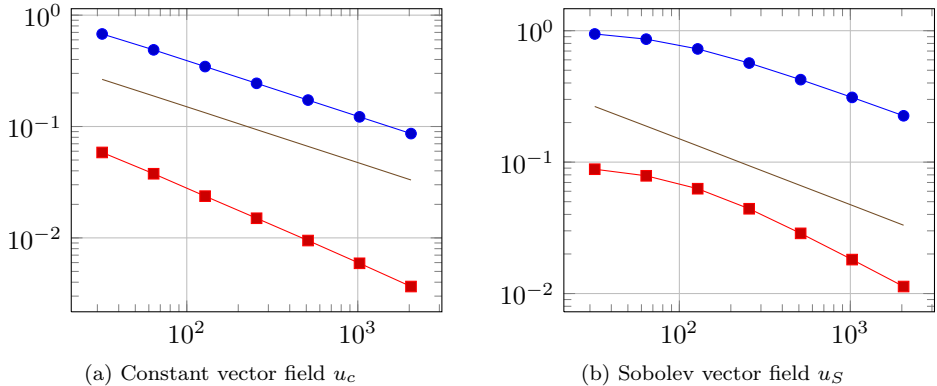
\begin{figure}[!b]\label{fig2}
\centering
\subfloat[Constant vector field $u_c$]{\begin{tikzpicture}
\begin{loglogaxis}[width=0.5\textwidth, grid=major]
\addplot table [x=meshsize,y=L1]{data_constant.csv};\label{p:const:L1}
\addplot table [x=meshsize,y=H-1]{data_constant.csv};\label{p:const:H-1}
\addplot+[mark=] table [x=meshsize,y=Rate]{data_constant.csv};\label{p:const:rate}
\end{loglogaxis}
\end{tikzpicture}}
\hspace{1.5em}
\subfloat[Sobolev vector field $u_S$]{\begin{tikzpicture}
\begin{loglogaxis}[width=0.5\textwidth, grid=major]
\addplot table [x=meshsize,y=L1]{data_sobolev.csv};
\addplot table [x=meshsize,y=H-1]{data_sobolev.csv};
\addplot+[mark=] table [x=meshsize,y=Rate]{data_constant.csv};
\end{loglogaxis}
\end{tikzpicture}}
\caption{Computation of the numerical error versus $1/h$ measured in terms of the $L^1$ norm (\ref{p:const:L1}) and the $H^{-1}$ norm (\ref{p:const:H-1}); as a reference we display the line $h^{1/2}$ (\ref{p:const:rate}).}
\end{figure}
Figure \ref{fig2} shows the computation of the final numerical error versus $1/h$ measured both in the $L^1$ norm and the (homogeneous) $H^{-1}$ norm.
The $H^{-1}$ norm metrizes weak convergence just as the Kantorovich--Rubinstein distance, though both measures are in general not equivalent. Indeed, on the one hand, Kantorovich--Rubinstein distances with concave cost functions are often bounded by the $H^{-1}$ norm. In our case, the sublinearity of the logarithm implies the bound
\[
\inf_{\pi}\iint \log\left(\frac{|x-y|}r+1\right)\, d\pi \le \inf_{\pi} \iint \frac{|x-y|}r\, d\pi  = \frac1r \|\rho-\rho_h\|_{\dot W^{-1,1}} \le \frac1r\|\rho-\rho_h\|_{\dot H^{-1}},
\]
because $|\Omega|=1$. See also \cite[Lemma 1]{Seis13b}.
Here, $\|\rho-\rho_h\|_{\dot W^{-1,1}}$ in its primal representation is also know as the Wasserstein distance $W_1(\rho,\rho_h)$ with cost function $c(x,y):=\abs{x-y}$. On the other hand, if $\rho_h$ is converging to $\rho$ weakly (in the sense of measures), then the Fourier coefficients $\mathcal{F}(\rho-\rho_{h})(k) = \int_{[0,1)^2} e^{-ix\cdot k} (\rho-\rho_{h})(x)\, dx$ are vanishing pointwise. In particular, using the Fourier representation of the $H^{-1}$ norm in the periodic setting, it holds for any $K\in 2\pi \N$ that
\[
\|\rho-\rho_h\|_{\dot H^{-1}}^2  = \ \ \sum_{\mathclap{k\in (2\pi \Z)^2}}\ \ |k|^{-2} |\mathcal{F} (\rho-\rho_h)(k)|^2 \le \ \ \sum_{\mathclap{|k|\le  K}}\ \  |k|^{-2} |\F(\rho-\rho_h)(k)|^2 +  K^{-2}\|\rho-\rho_h\|_{L^2}.
\]
Because $\rho$ and $\rho_h$ are both bounded in $L^2$, by choosing $K$ sufficiently large, the latter shows that $\rho_h$ is converging to $\rho$ in $H^{-1}$.

In our numerical tests, we have chosen the $H^{-1}$ norm over the Kantorovich--Rubinstein distance as the latter is particularly easy to compute numerically. 
The computation in Figure \ref{fig2} shows that in both experiments, the numerical error does not exceed the order $1/2$ in the regime of small mesh sizes. Moreover, we observe that the $H^{-1}$~decay is slightly steeper than the $L^1$ decay. We interpret this feature with a certain enhanced ``mixing effect'' caused by the scheme with rough vector fields .

\section{Properties of the continuous model}\label{S:3}

If $u: [0,T]\times \Omega\to \R^d$ is a smooth vector field on the bounded Lipschitz domain $\Omega$ in $\R^d$, then the {\em flow of u} is the mapping $\phi: [0,T]\times \Omega\to \R^d$ solving the ordinary differential equation
\begin{equation}
\label{e:flow:sol}
\partial_t \phi(t,x) = u(t,\phi(t,x)),\quad \phi(0,x) = x.
\end{equation}
Thus $t\mapsto \phi(t,x)$ is the trajectory of a particle transported by $u$ and starting at $x\in\Omega$. The condition that $u$ is tangential at the boundary, see \eqref{12}, guarantees that there is no flow out of the domain.

In the present paper, we consider vector fields under low regularity assumptions. We recall from \eqref{e:ass:u} that $u$ is uniformly bounded and weakly differentiable in the spatial variable, and it is
{\em nearly incompressible} in the sense that
\begin{equation}
\label{10}
\lambda:= \|(\div u)^-\|_{L^1((0,T);L^{\infty}(\Omega))} <\infty.
\end{equation}
Under these assumptions, a generalized notion of a solution of \eqref{e:flow:sol}  is needed:
A mapping $\phi: [0,T]\times \Omega\to \R^d$ is called a {\em regular Lagrangian flow}, if
\begin{enumerate}
\item for a.e.\ $x\in\Omega$, the mapping $t\mapsto \phi_t(x):=\phi(t,x)$ is an absolutely continuous integral solution, i.e.,
\[
\phi(t,x) = x  + \int_0^t u(s,\phi(s,x))\, ds\quad\mbox{for all }t\in[0,T];
\]
\item there exists a constant $\Lambda$ independent of $t \in [0,T]$ such that
\[
 \L^d(\phi_t^{-1}(A)) \le \Lambda \L^d(A)\quad\mbox{for any Borel subset $A$ of $\Omega$}.
\]
The constant $\Lambda$ is often called the compressibility constant of $\phi$.
\end{enumerate}
Existence, uniqueness and stability of regular Lagrangian flows in the setting of our paper have been proved by DiPerna and Lions \cite{DiPernaLions89} (for vector fields with bounded divergence), see also \cite{CrippaDeLellis08} for a quantitative approach (under the milder assumption \eqref{10}).

The compressibility assumption \eqref{10} implies that the generalized Jacobian
\begin{equation*}
J\phi_t(x)  = \exp\left(\int_0^t (\div u)(s,\phi_t(x))\, ds\right)
\end{equation*}
is bounded below,
\begin{equation}
\label{14}
e^{-\lambda} \le J\phi_t(x),
\end{equation}
and for any $f\in L^1(\Omega)$ we have thanks to the boundary condition~\eqref{12} the change of variable formula
\begin{equation}
\label{15}
\int_{\Omega} f(\phi_t(x)) J\phi_t(x)\, dx = \int_{\Omega} f(x)\, dx.
\end{equation}
 Notice that we may choose $\Lambda=e^{\lambda}$ in the definition of regular Lagrangian flows. A comprehensive analysis of the generalized Jacobian can be found in \cite{ColomboCrippaSpirito15}.

With the help of regular Lagrangian flows, solutions to the continuity equation~\eqref{1} take on an elegant form. Indeed, if $\rho$ is the unique solution with initial datum $\rho_0$, we may simply write $\rho(t,\tacka) = (\phi_t)_{\# }\rho_0$, where $\#$ denotes the push forward operator, defined by
\begin{equation*}
\int_\Omega f \, d (\phi)_\# \mu = \int_\Omega f\circ \phi\, d\mu,
\end{equation*}
for any Borel function $f$ and any Borel measure $\mu$. Thanks to the change of variables formula \eqref{15}, we thus have the identity  $\rho(t,\phi_t(\cdot))\, J\phi_t(\cdot) = \rho_0(\cdot)$.  Therewith, we can estimate
\[
\|\rho(t,\cdot)\|_{L^q}^q \stackrel{\eqref{15}}{=} \int_{\Omega}|\rho(t,\phi_t(x))|^q J\phi_t(x)\, dx \stackrel{\eqref{14}}{\le} \Lambda^{(q-1)} \|\rho_0\|_{L^q}^q,
\]
and thus
\begin{equation}
\label{e:Phi:Lq}
\|\rho\|_{L^{\infty}(L^q)} \le \Lambda^{1-\frac1q} \|\rho_0\|_{L^q}.
\end{equation}
The relation between the continuity equation \eqref{1} and the ordinary differential equation~\eqref{e:flow:sol} is reviewed in \cite{AmbrosioCrippa14}.

\section{Transport distance with logarithmic cost}\label{S:5}

In this section we review some properties of transport distances with  logarithmic cost functions. For a comprehensive introduction to the theory of optimal transportation, we refer to Villani's monograph~\cite{Villani03}.

Given two nonnegative distributions $\rho_1$ and $\rho_2$ of the same mass $\rho_1[\Omega] = \rho_2[\Omega]$,
a \emph{transport plan} or \emph{coupling} $\pi$ is a plan that determines how the distribution $\rho_1$ is transferred to the distribution $\rho_2$. These are characterized by the condition
\[
\pi[A\times \Omega] = \int_A\rho_1\, dx,\quad \pi[\Omega\times A] = \int_A\rho_2\, dx\qquad\mbox{for all measurable }A\subset \Omega .
\]
In this paper, we will rather use the equivalent characterization
\begin{equation}
\label{18}
\int \bra[\big]{\varphi(x) + \xi(y)} \, d\pi(x,y) = \int \varphi \rho_1\, dx + \int \xi \rho_2\, dx\qquad\text{for all }\varphi,\xi\in C(\overline\Omega).
\end{equation}
The set of all transport plans between $\rho_1$ and $\rho_2$ will be denoted by $\Pi(\rho_1,\rho_2)$.

The problem of optimal transportation is to minimize the total cost that is necessary for transferring configuration $\rho_1$ into configuration $\rho_2$. Here, we will always assume that costs are measured relative to the distance of shipment. Given a nonnegative cost function $c$ on $[0,\infty)$, this amounts to minimizing
the total transport cost
\[
\iint c(|x-y|)\, d\pi(x,y)
\]
among all admissible transport plans $\pi\in \Pi(\rho_1,\rho_2)$.

In most parts of this paper, we will consider logarithmic cost functions $c(z) = \log(z/r+1)$ with some positive parameter $r$, and we write
\[
\D_r(\rho_1,\rho_2) = \inf_{\pi\in \Pi(\rho_1,\rho_2)} \iint \log\left(\frac{|x-y|}r +1\right)\, d\pi(x,y).
\]
As any concave function, the logarithmic cost function induces a metric on $\Omega$ by setting $d(x,y) = c(|x-y|)$. This crucial insight has a number of important consequences that we gather in the following.

\begin{itemize}
\item The minimal total cost $\D_r$ constitutes a distance on the space of densities with equal mass on $\Omega$, cf.\ \cite[Theorem 7.3]{Villani03}. In particular, it obeys the triangle inequality for all densities $\rho_1,\rho_2,\rho_3 \in L^1(\Omega)$,
\begin{equation}
\label{19}
\D_r(\rho_1,\rho_2)\le \D_r(\rho_1,\rho_3) + \D_r(\rho_3,\rho_2).
\end{equation}
In the literature, such distances go by difference names, including ``Wasserstein distance'', ``Monge--Kantorovich distance'' or ``Kantorovich--Rubinstein distance'' depending on the context and the mathematical community. In this paper, we will choose the third of these options, as motivated by the following observation:
\item It holds the Kantorovich--Rubinstein dual representation for any $\rho_1,\rho_2\in L^1(\Omega)$
\[
\D_r(\rho_1,\rho_2) = \sup_\psi \set*{ \int \psi  \, d\bra{\rho_1 - \rho_2} : \abs{\psi(x)-\psi(y)} \leq \log\bra*{\frac{\abs{x-y}}{r} + 1}} ,
\]
cf.\ \cite[Theorem 1.14]{Villani03}. In particular, the transport distance between two densities~$\rho_1$ and~$\rho_2$ only depends on their difference $\rho_1-\rho_2$ and it holds the \emph{transshipment identity}
\begin{equation}
\label{20}
\D_r(\rho_1,\rho_2)  = \D_r((\rho_1-\rho_2)^+,(\rho_1 - \rho_2)^-).
\end{equation}
Moreover, the latter allows for extending the definition of the Kantorovich--Rubin\-stein distance to any two not necessarily nonnegative densities of same mass.
\item Kantorovich--Rubinstein distances defined for densities on a compact domain~$\Omega$ metrize weak convergence for measures, i.e.,
\[
\D_{r}(\rho_h,\rho) \to 0\quad \Longleftrightarrow\quad \int f \, d\rho_h \to \int f \, d\rho\text{ for all $f\in C(\overline\Omega)$ as } h\to 0,
\]
see \cite[Theorem 7.12]{Villani03} for more details.
\item The minimum in the primal and the maximum in the dual formulation are both attained, see, for instance, Exercise 2.36 and Theorem 2.45 in \cite{Villani03}. We will denote the optimal transport plan by $\pi_{\opt}$. A characterization of $\pi_{\opt}$ was given by Gangbo and McCann \cite{GangboMcCann96}, see also \cite[Theorem 2.45]{Villani03}.
\item It has been shown in \cite{Seis16a} that the mapping
$t\mapsto \D_r(\rho_1(t),\rho_2(t))$ is absolutely continuous with derivative
  \begin{equation}\label{e:Dist:deriv}
    \abs*{\frac{d}{dt} \D_r(\rho_1(t),\rho_2(t))} \leq \iint\frac{|u(x)-u(y)|}{|x-y|+r}\, d\pi_{\opt}(x,y),
  \end{equation}
if $\rho_1$ and $\rho_2$ are two integrable distributional solutions of the continuity equation $\partial_t \rho + \div\bra{u \, \rho} = 0$. In \eqref{e:Dist:deriv}, $\pi_{\opt} = \pi_{\opt}(t)$ is the optimal transport plan for $\D_r((\rho_1(t)-\rho_2(t))^+,(\rho_1(t)-\rho_2(t))^-)$.
\end{itemize}
We conclude this section with a particular coupling $\pi\in \Pi(\rho_1,\rho_2)$, which applies to the case where the densities can be written as the push forward of the same density under different flows, $\rho_1 = \psi_{\#}\rho$, $\rho_2 = \phi_{\#}\rho$.
More generally, let $(\psi(x))_{x\in \Omega}$ and $(\phi(x))_{x\in \Omega}$ be two families of random variables on the common standard probability space $\bra*{\Omega, \cF,\PP}$. Let $\EE$ denote the corresponding expectation and $\PP_x\bra*{\psi \in dy}:=\PP\bra*{\psi(x) \in dy}$. Then for any nonnegative $\rho$, we have $\bra*{\psi_\# \rho}(dy) = \int \PP_x\bra*{ \psi \in dy} \, \rho(dx)$ and it holds
\begin{equation}\label{e:standard_coupling}
  \D_r(\psi_\# \rho, \phi_\# \rho) \leq \int \EE_x\pra*{\log\bra*{\frac{|\psi - \phi|}{r} + 1}} \, \rho(dx) =: \EE_{\rho}\pra*{\log\bra*{\frac{|\psi - \phi|}{r} + 1}},
\end{equation}
In the following, we will refer to this coupling as the \emph{standard coupling}.
In particular, we can apply the standard coupling to the regular Lagrangian flow $\phi_t$ from Section~\ref{S:3}, which is then interpreted as a random variable with $\PP_x\bra{ \phi_t \in dy } = \delta_{\phi_t(x)}(dy)$.

\section{Properties of the Upwind scheme}\label{S:4}

In this section, we derive some intrinsic properties of the upwind finite volume scheme. Except noted otherwise, all these properties will be valid for any (unstructured) tessellation $\cT$, for which the regularity condition \eqref{e:trace} is active.

\subsection{Basis properties}

Let us start by discussing the relation between the two CFL conditions~\eqref{e:CFL_num} and~\eqref{CFL_true}. The following lemma shows, that~\eqref{CFL_true} is sufficient for~\eqref{e:CFL_num} provided the implicit constant is chosen small enough.
\begin{lemma}[Verification of CFL condition~\eqref{e:CFL_num}]\label{lem:CFL}
Suppose that the mesh $\cT$ satisfies~\eqref{e:trace} with constant $C_0$  and that $u\in L^1(L^\infty \cap W^{1,p})$. Then it holds for any $K\in \cT$ and all $n\in \llb 0,  N \rrb$ that
\begin{equation}\label{e:est:CFL}
  \sum_{L\sim K} p_{KL}^n = \delta t \, \sum_{L\sim K} \tau_{KL} u_{KL}^{n+} \leq C_0 \frac{\delta t}{h} u^n_\infty,
\end{equation}
where $u^n_\infty$ is defined in~\eqref{CFL_true}.
\end{lemma}
\begin{proof}
By using the definition $\tau_{KL}=\abs{K\edge L}/\abs{K}$ and  the one of $u_{KL}^{n+}$ in \eqref{e:def:uKsigma}, we arrive at
\begin{align*}
  \delta t  \sum_{L\sim K} \tau_{KL} u_{KL}^{n+} = \frac{\delta t}{\abs{K}} \sum_{L\sim K} \bra*{ \avint_{t^n}^{t^{n+1}} \int_{K\edge L} u \cdot \nu_{KL} \,d\Ha^{d-1}\,dt }^+.
\end{align*}
Now, since $u\in L^1(L^\infty\cap W^{1,p})$, its time average $\avint_{t^n}^{t^{n+1}} u \, dt$ is in $L^\infty\cap W^{1,p}$ and in particular in $L^\infty\cap W^{1,1}$, since $\Omega$ is bounded. Hence, we can apply a trace estimate and obtain
\begin{align*}
  \delta t  \sum_{L\sim K} \tau_{KL} u_{KL}^{n+} \leq \delta t \ \avint_{t^{n}}^{t^{n+1}} \norm{u}_{L^\infty} \, dt \  \sum_{L\sim K} \frac{\abs{K\edge L}}{\abs{K}} = \delta t \, u^n_\infty \, \frac{\abs{\partial K}}{\abs{K}}  \leq C_0 \frac{\delta t}{h} u^n_\infty,
\end{align*}
where the last step follows by the regularity property \eqref{e:trace} (or~\eqref{e:isoper}).
\end{proof}

In the next lemma, we summarize classical properties of the upwind scheme. These are monotonicity, mass preservation and stability. Notice that for the derivation of these properties, the numerical CFL condition \eqref{e:CFL_num} is sufficient.

\begin{lemma}[Stability estimates]
\label{lem:prop_scheme}
The upwind finite volume scheme has the following properties:
  \begin{enumerate}[ (i) ]
   \item If $\rho^0_h\geq 0$, then $\rho_h^n\geq 0$ for all $n\in \llb 0, N \rrb$.
   \item For any $n\in \llb 0, N \rrb$ it holds
   \begin{equation}\label{e:mas:conserve}
     \int \rho_h^n\, dx = \int \rho_h^0\, dx.
   \end{equation}
   \item For any $q\geq 1$ it holds
   \begin{equation}\label{e:scheme:Linfty}
      \| \rho_h \|_{L^{\infty}(L^q)} \le  \exp^{1-\frac1q}\bra*{ \int_0^{T} \norm{ (\div u)^-}_{L^\infty} dt }\norm{\rho^0_h}_{L^q} = \Lambda^{1-\frac1q} \norm{\rho^0_h}_{L^q}.
   \end{equation}
  \end{enumerate}
\end{lemma}
\begin{proof}
  Let us first note, that under the CFL condition~\eqref{e:CFL_num} it holds
  \[
     1- \delta t \sum_{L\sim K} \tau_{KL} u_{KL}^{n+} \geq 0 ,
  \]
  and therefore $0\le p_{LK}^n\le 1$  for all $K,L\in \cT$.
  As a consequence, we deduce from \eqref{e:upwind2} that $\rho_K^{n+1}$ is defined as a conical combination of $\bra{\rho_L^n}_{L\in \cT}$, which in turn implies \emph{(i)}.

  Summation over $K\in\cT$ in the upwind scheme \eqref{e:upwind2} results in
  \[
    \sum_{K\in \cT} \abs{K} \rho_K^{n+1} = \sum_{K,L\in \cT} \abs{L} p_{LK}^n \rho_L^n = \sum_{L\in \cT} \abs{L} \rho_L^n \sum_{K\in\cT} p_{LK}^n ,
  \]
  which implies~\eqref{e:mas:conserve}, since $\sum_{K} p_{LK}^n=1$. This proves \emph{(ii)}. Moreover, by applying the modulus and the triangle inequality to the above identity, we obtain the case $q=1$ in the estimate of \emph{(iii)}.

  Now, we prove the other pivotal estimate for $q=\infty$. At first, we calculate
  \begin{align*}
    \frac{1}{\abs{K}} \sum_{L\in\cT} \abs{L} p_{LK}^n
    &= 1 - \delta t \sum_{L\sim K} \frac{\abs{K \edge L}}{| K | } \bra{ u_{KL}^{n+} - u_{KL}^{n-}} \\
    &= 1- \frac{1}{\abs{K}} \int_{t^n}^{t^{n+1}} \int_{\partial K } u \cdot \nu\,d\Ha^{d-1}dt\\
    & = 1- \int_{t^n}^{t^{n+1}} \avint_K \div u \,dx\,dt.
  \end{align*}
  Then, applying the absolute value and taking the maximum in \eqref{e:upwind2}, we find that
  \[
    \abs{\rho_K^{n+1}} \le\bra[\bigg]{\frac{1}{\abs{K}}\sum_{L\in\cT} \abs{L} p_{LK}^n} \max_{K\in\cT} \abs{\rho_K^n} \le \bra[\bigg]{1+\int_{t^n}^{t^{n+1}} \norm{\bra{\div u}^-}_{L^\infty(\Omega)} \, dt} \max_{K\in \cT} \abs{\rho_K^{n}}.
  \]
  After passing to the supremum in the spatial variable we thus have
  \[
    \max_{K\in\cT} \abs{\rho_K^{n+1}} \le \bra[\bigg]{1+\int_{t^n}^{t^{n+1}} \norm{\bra{\div u}^-}_{L^\infty(\Omega)} \, dt } \max_{K\in\cT} \abs{\rho_K^n}.
  \]
  The statement in \eqref{e:scheme:Linfty} with $q=\infty$ then follows by iteration thanks to the elementary inequality $\Pi_{n}( 1+a_n ) \le \exp\left(\sum_n a_n\right)$.

  Summarizing the previous two steps, we have found that the upwind scheme at time step~$n$ defines a bounded linear operator from $L^1$ to $L^1$ with norm bounded by $1$, and from $L^{\infty}$ to $L^{\infty}$ with norm bounded by $\exp\bra[\big]{\int_0^{t^n} \|(\div u)^-\|_{L^{\infty}}\, dt}$. The Riesz--Thorin interpolation theorem then yields the result for any $q\in[1,\infty]$.
\end{proof}

\subsection{Probabilistic interpretation}

We will see in Section \ref{S:6} below that it is enough to consider configurations that are \emph{nonnegative}. Note that this is consistent with the upwind scheme by the first property in Lemma \ref{lem:prop_scheme} above. We will thus assume from here on that $\rho_K^n\ge 0$ for all $n \in \llb 0,N\rrb$ and all $K\in\cT$.

Following the ideas of Delarue, Lagouti\`ere and Vauchelet in \cite{DelarueLagoutiere11,DelarueLagoutiereVauchelet16}, we associate random characteristics with the upwind scheme.
Therefore, we construct a Markov chain~$(J^n)_{n\in \N_0}$ with state space $\cT$.
We use $\cT^{\N_0}$ as the canonical space and $(J^n)_{n\in \N_0}$ is then the  canonical process.
The $\sigma$-field $\cA$ is generated by sets $\prod_{n\in \N} A^n$ with $A^n\subset \cT$ and $A^n=\cT$ for any sufficiently large value of $n$.
The canonical filtration is $\FF=\bra{\cF^n = \sigma(J^0,\dots, J^n)}_{n\in \N_0}$. We endow $(\cT^{\N_0},\cA)$ with a collection of probability measures $(\PP_K)_{K\in \cT}$. Here, the element $K\in \cT$ is the initial point for the process $(J^n)_{n\in\N_0}$, i.e., it holds $\PP_K(J^0 = L) = \delta_{K,L}$. Moreover, $\PP$ defines the Markov chain $(J^n)_{n\in \N_0}$ with transition matrix $(p_{KL}^n)_{K,L \in \cT}$ as defined in~\eqref{e:def:p_KL}.
We thus have the relation
\begin{equation*}
  \PP( J^{n+1} = L \mid \cF^n ) = p_{J^n L}^n.
\end{equation*}
For $\mu$, a nonnegative measure on $\cT$, we define $\PP_{\mu}$ by
\begin{equation*}
  \PP_{\mu}(\tacka) = \sum_{K\in \cT} \mu(K) \PP_K(\tacka).
\end{equation*}
This is a Markov chain starting from $\mu$, i.e., $\PP_\mu(J^0 = K) = \mu(K)$. We will denote by $J^n_\# \mu$ the law of the Markov chain at time $n$ started from $\mu$.

In the following lemma, $\rho^n$ denotes the solution of the upwind scheme at time $n\in\N_0$ considered as a vector indexed by the control volumes $K\in\cT$.

\begin{lemma}[Flow representation of approximate solution]
The solution of the upwind scheme is the pushforward of the discretized initial datum by the Markov chain, i.e., $\rho^n = J^n_\# \rho^0$ for all $n\in\N$.
\end{lemma}
\begin{proof}
  The proof is identical to~\cite[Lemma 3.6]{DelarueLagoutiereVauchelet16}. Just recall that~\eqref{e:upwind} can be rewritten in the form \eqref{e:upwind2}, which gives the evolution of the law of the Markov chain~$J^n$.
\end{proof}

We now define random characteristics $(\psi^n)_{n\in \N_0}$ as a Markov chain with state space $\Omega\subset \R^d$.
We use the canonical space $\Omega^{\N_0}$ with $\sigma$-algebra $\cB$ generated by sets $\prod_{n\in\N_0} B^n$, where $\bra{B^n}_{n\in\N_0}$ is the family of Borel sets in $\R^d$ such that $B^n= \R^d$ for any sufficiently large $n$ . We will use two filtrations. Firstly, the canonical filtration $\FF_{\R^d} = \bra[\big]{\cF^n_{\R^d} = \sigma(\psi^0,\dots,\psi^n)}_{n\in \N_0}$ and secondly the just defined coarse filtration $\FF=\bra{\cF^n = \sigma(J^0,\dots, J^n)}_{n\in \N_0}$, where $J^k \in \cT$ is the corresponding Borel measurable cell $J^k\subset \R^d$.

Then, we endow $(\Omega^{\N_0}, \cB)$ with a family of probability measures $\bra{\PP_z}_{z\in \Omega}$ generating a continuous-state Markov chain started in $z$, whose jump kernel at $x\in K$ is given by
\begin{equation}\label{e:JumpKernel}
  q^n(x,dy) := \begin{cases}
                 p^n_{KL} \, \frac{dy}{\abs{L}} &\mbox{for } dy\in L \sim K, \\
                 p^n_{KK} \, \delta_x(dy) &\mbox{for } dy\in K ,
               \end{cases}
\end{equation}
where $p^n_{KL}$ is defined in~\eqref{e:def:p_KL}.
Hence, it holds
\begin{equation}\label{e:def:Psi}
  \PP_z(\psi^{n+1} \in dy \mid \cF^n_{\R^d}) = q^n(\psi^n, dy) .
\end{equation}
For any $\FF_{\R^d}$-measurable nonnegative $\mu$, we set $\PP_\mu(\tacka) = \int \PP_z(\tacka) \,\mu(dz)$. In the following, we will exclusively start the Markov chain from $\FF$-measurable densities like the discretized initial data $\rho_h^0$ as defined in \eqref{e:disc_initial}.

Let us collect some properties of the random characteristics and point out the close links between the Markov chains $\bra{J^n}_{n\in\N_0}$ and $\bra{\psi^n}_{n\in\N_0}$.

\begin{lemma}\label{lem:ave:Psi}
Suppose that $\rho_h^0$ is nonnegative. Then the following holds:
  \begin{enumerate}[ (i) ]
   \item The distribution $\psi^n_\# \rho^0_h$ is constant on each cell $K\in \cT$. Moreover, if $\rho_{h}^n$ is given by~\eqref{e:rho_h} then $\rho_{h}^n = \psi^n_\# \rho^0_h$. In particular, it holds
\begin{equation}\label{e:Psi:J}
  \PP_{\rho^0_h}\bra{ \psi^n \in \, dx \mid J^n } = \frac{\chi_{J^n}(x)\, dx}{\abs{J^n}}.
\end{equation}
  Hence $\psi^n$ is uniformly distributed over any control volume $K\in\cT$, and thus $\PP_{\rho^0_h}(\psi^n\in K) = \PP_{\rho^0_h}(J^n = K )$ .
  \item If $\EE_{\rho^0_h}$ denotes the expectation under the law $\PP_{\rho^0_h}$, then
   \begin{equation}\label{e:EE:Psi}
     \EE_{\rho^0_h}[ \psi^{n+1}- \psi^{n} \mid \psi^n ] = \delta t \, u^n_h(\psi^n) ,
   \end{equation}
   where the net flow $u^n_h(x)$ for $x\in K$ is defined by
   \begin{equation}\label{e:def:netflow}
     u^n_h(x) := u^n_K := \sum_{L\sim K} p_{KL}^n \frac{x_L - x_K}{\delta t} \qquad\text{with}\qquad x_K := \avint_K x \, dx .
   \end{equation}
   If $K$ is a control volume of a Cartesian tessellation, then $u_K^n$ is related to the net outflow over the edges by the identity
   \begin{equation}\label{e:netflow:cubes}
      u^n_K  = \sum_{L\sim K} \nu_{KL} u_{KL}^{n+} .
   \end{equation}
   \item The relation
   \begin{equation*}
     \xi^n := \psi^{n+1} - \psi^n  - \EE\pra{ \psi^{n+1} - \psi^n \mid \cF^n }
   \end{equation*}
   defines a family of $\cF^{n+1}_{\R^d}$ measurable random variables $\bra*{\xi^n\in \R^d}_{n\in \N_0}$
   satisfying
   \begin{equation}\label{e:est:xi}
     \EE[\xi^n \mid \cF^{n}] = 0 \qquad\text{and}\qquad \abs{\xi^n} \leq 4h \quad\text{a.s.}
   \end{equation}
   Moreover, for any $m\geq 1$, there exists a constant $C>0$ only depending on $m$ such that
   \begin{equation}\label{e:est2:xi}
     \EE\pra[\big]{\abs{\xi^n}^m \mid \cF^{n}} \leq C \, \delta t \, u^n_\infty \, h^{m-1},
   \end{equation}
   where $u^n_\infty$ is defined in~\eqref{CFL_true}.
  \end{enumerate}
\end{lemma}

Formula \eqref{e:netflow:cubes} is the only place in this paper in which the assumption that $\cT$ is a Cartesian triangulation will be used.  We will further comment on this on page \pageref{mesh_discussion} below.
\begin{proof}
First, we verify that $\psi^n_\#\rho_h^0 $ is constant on each cell $K$ and that the constant satisfies the same recursion as the solution of the upwind scheme in~\eqref{e:upwind2}. We note that $\rho^0_h$ by definition~\eqref{e:disc_initial} is absolutely continuous with respect to the Lebesgue measure and by construction of the jump kernel~\eqref{e:JumpKernel} so is $\psi^{n}_\# \rho_h^0$ for any $n\geq 1$. Therefore, we let $\varphi$ be a continuous test function on $\overline\Omega$ and calculate
\begin{align*}
  \int \varphi(x) \, \psi^{n+1}_{\#} \rho^{0}_h(dx)
  &\stackrel{\mathclap{\eqref{e:def:Psi}}}{=} \iint \varphi(x) \, q^n(y,dx) \,  \psi^{n}_{\#} \rho^{0}_h(dy) \\
  &= \sum_{L\in\cT} \int_L \bra*{\sum_{K\sim L} p_{LK}^n \avint_K \varphi(x) dx  + \varphi(y) p_{LL}^n } \psi^{n}_{\#}\rho^{0}_h(dy).
\end{align*}
To show, that $\psi^{n+1}_{\#}\rho^0_h$ is constant on any $K\in \cT$, we argue by induction. Since $\psi^0_{\#} \rho^{0}_h = \rho^0_h$ the base case is settled. Now we assume that $\psi^{n}_{\#} \rho^{0}_h$ is constant on every $K\in \cT$ and  we denote these constants by $c_K^n$. Then we obtain
\[
  \int \varphi(x) \, \psi^{n+1}_{\#} \rho^{0}_h(dx) =  \sum_{L\in\cT} \sum_{K\in\cT}  \abs{L} c_L^n \, p_{LK}^n \; \avint_K \varphi(x) \, dx  .
\]
Now, if $\varphi$ is such that $\avint_K \varphi(x) dx=0$ for all $K\in \cT$, the right-hand side vanishes. Hence, also $\psi^{n+1}_{\#} \rho^{0}_h(dx)$ is constant on any $K\in \cT$.

As a consequence, by choosing $\varphi$ as the characteristic function for some fixed $ K\in \cT$ (which can be done by approximation), we infer the identity
\[
  |K| c_K^{n+1} = \sum_{L\in\cT} |L| c_L^n \, p_{LK}^n .
\]
Comparing this formula with \eqref{e:upwind2} and recalling that $c_K^0  = \rho_K^0$, the uniqueness of the explicit scheme yields that $c_K^n = \rho_K^n$ for any $n$. Hence $\rho_h^{n+1} = \psi^{n+1}_\#\rho_h^0$.

For the second property, we calculate for $K\in \cT$ fixed:
  \begin{align*}
     \EE[ \psi^{n+1}- \psi^{n} \mid \psi^n \in K ]
     &\ \underset{\mathclap{\eqref{e:Psi:J}}}{\overset{\mathclap{\eqref{e:def:Psi}}}{=}}\  \frac{1}{\abs{K}} \int_K \int (y-x) \, q^n(x,dy) \, dx \\
     &\ \overset{\mathclap{\eqref{e:JumpKernel}}}{=}\ \sum_{L\sim K} p_{KL}^n \ \avint_K\avint_L \bra{y - x } \, dy\,dx \\
     &\ =\  \sum_{L\sim K} p_{KL}^n ( x_L - x_K ) = \delta t \, u_K^n.
  \end{align*}
  In case of uniform rectangular meshes, we have for all $L\in \cT$ the identity $\abs{ K \edge L } \, \abs{x_L - x_K} = \abs{K}$ and can further rewrite
  \begin{align*}
    u_K^n &\stackrel{{\eqref{e:def:p_KL}}}{=}  \sum_{L\sim K} \frac{\abs{K \edge L} \, \bra{x_L - x_K}}{\abs{K}} u_{KL}^{n+}   = \sum_{L\sim K} \nu_{KL} u_{KL}^{n+} .
\end{align*}

We now turn to the proof of \emph{(iii)}. The measurability and mean-zero property of~$\xi^n$ follow immediately from the latter's definition. In view of \eqref{e:EE:Psi} and \eqref{e:def:netflow}, the norm of $\xi^n$ can be estimated by
\[
  \abs{\xi^n} \leq \abs{\psi^{n+1} - \psi^n} + \delta t \, \abs{u^n_{J^n}} \leq 2h + 2 h \sum_{L\sim K} p_{KL}^n \leq 4 h .
\]
For the further characterization of $\xi^n$, we calculate using a test function $\varphi$ on $\R^d$,
\begin{align*}
  \EE\pra{ \varphi(\xi^n) \mid \cF^n } &\ \underset{\mathclap{\eqref{e:Psi:J}}}{\overset{\mathclap{\eqref{e:def:Psi}}}{=}}\  \avint_{J^n} \int \varphi\bra*{y-  x -\delta t\, u_{J^n}^n}\, q^n(x ,dy) \, dx \\
  &\ \overset{\mathclap{\eqref{e:JumpKernel}}}{=}\ p_{J^n J^n}^n \, \varphi(-\delta t \, u_{J^n}^n) +  \sum_{L\sim J^n} p_{J^n L}^n \; \avint_{J^n} \avint_L \varphi\bra*{y- x - \delta t \, u_{J^n}^n} \, dy \, dx .
\end{align*}
Choosing $\varphi(x) = \abs{x}^m$ as a test function, we obtain the estimate (in which we suppress the $m$-depending constant)
\begin{align*}
  \EE\pra[\big]{ \abs{\xi}^m \mid \cF^n } &\lesssim  h^m\sum_{L\sim J^n} p_{J^n L}^n + \abs{\delta t\, u_{J^n}^n}^m \lesssim h^m \sum_{L\sim J^n} p_{J^n L}^n ,
\end{align*}
where we used the estimate
\[
  \abs{\delta t\, u_{J^n}^n}^m \overset{\eqref{e:def:netflow}}{\lesssim} \bra[\bigg]{h \sum_{L\sim J^n} p_{J^n L}^n}^m \leq h^m \sum_{L\sim J^n} p_{J^n L}^n  .
\]
The conclusion now follows from~\eqref{e:est:CFL}.
\end{proof}

From part \emph{(iii)} of the above lemma it follows that the random characteristics satisfy the discrete difference equation
\begin{equation}\label{e:Psi:difference}
  \psi^{n+1}-\psi^n = \delta t \, u_h^n(\psi^n) + \xi^n ,
\end{equation}
which is a time and space discretized variant of the stochastic differential equation~\eqref{intro:SDE}. We therefore expect that the martingale part $\sum_{\ell=0}^n \xi^\ell$ behaves like a rescaled random walk of scale $h$. This conjecture will be confirmed by the following lemma, in which an $h^{1/2}$ bound on the martingale part is established.

\begin{lemma}[Martingale estimate]\label{L8}
Suppose that $\rho_h^0$ is nonnegative. 
  For any $1\leq r < \infty$ there exists a positive constant $C$ such that 
   \begin{equation}\label{e:est:quad:xi}
     \EE_{\rho^0_h}\pra*{\sup_{0\leq k < N } \abs[\bigg]{\sum_{n=0}^k \xi^n}^r}^{\frac{1}{r}} \leq C \bra*{\sqrt{h \norm{u}_{L^1(L^\infty)}} +  h} \norm{\rho^0_h}_{L^1}^\frac{1}{r} .
   \end{equation}
\end{lemma}
Our proof of Lemma \ref{L8} is (in parts) similar to the one of~\cite[Lemma 4.5]{DelarueLagoutiereVauchelet16}.
\begin{proof}
For the proof, we note that we can convert $\rho_h^0$ into a probability measure by normalizing  $\rho^0_h / \norm{\rho^0_h}_{L^1}$. Hence, we take the expectation  with respect to an initial $\FF$-measurable probability distribution, which we omit in the following part of the proof. Let $M_{k} := \sum_{\ell=0}^{k-1} \xi^{\ell}$ with $M_0 := 0$ and $M^*_k := \sup_{1\leq \ell \leq k} \abs{M_\ell}$. Then by \textit{(iii)} of Lemma~\ref{lem:ave:Psi}, it holds
\[
  \EE\pra[\big]{ M_k \mid \cF^{k-1} } = M_{k-1} + \underbrace{\EE\pra[\big]{\xi^{k-1} \mid \cF^{k-1}}}_{=0} = M_{k-1} ,
\]
and thus, $M_k$ is a discrete mean-zero martingale. This observation turns out to be crucial for the remainder of this proof as it enables us to apply two well-known martingale maximal inequalities to $(M_k)$. One of these is the Burkholder--Davis--Gundy inequality~\cite[Proposition 15.7]{Kallenberg2002}, which for any $r\in[1,\infty)$  takes the form
\[
 \EE\pra[\big]{ \bra{M_k^*}^r}  \leq C_r \, \EE\pra[\Big]{ \pra{M}_k^{\frac{r}{2}}} .
\]
Here, $\pra{M}_k$ is the quadratic variation of $\bra{M_k}$ given by
\begin{align*}
  \pra{M}_k &= \sum_{n=1}^{k} \abs{M_n - M_{n-1}}^2 = \sum_{n=0}^{k-1} \abs{\xi^{n}}^2,
\end{align*}
and $C_r$ is a constant dependent only on $r$. Note, that we have $\EE\pra[\big]{ \pra{M}_k^{\frac{r}{2}}}^{\frac{1}{r}}\leq \EE\pra[\big]{ \pra{M}_k}^\frac{1}{2}$ for any $r\in [1,2]$ thanks to Jensen's inequality. For these values of $r$, it is thus enough to consider the expectation of the quadratic variation, which by linearity and the law of total expectation becomes
\begin{align*}
  \EE\pra[\big]{\pra{M}_k} = \sum_{n=0}^{k-1} \EE\pra[\Big]{ \EE\pra*{ \abs*{\xi^{n}}^2 \,\middle|\, \cF^{k} }} \ \overset{\mathclap{\eqref{e:est2:xi}}}{\leq}\ C h \, \delta t \sum_{n=0}^{k-1} u^{n}_\infty.
\end{align*}
This estimate gives the conclusion in the case $r\in[1,2]$.

We furthermore notice that in expectation, the quadratic variation is equal to the $L^2$ norm of $M_n$ as a consequence of the orthogonality of $\bra{\xi^n}$, cf.\ \eqref{e:est:xi}, and the law of total expectation. The previous estimate thus implies
\begin{equation}\label{e:est:MnL2}
  \EE\pra[\big]{|M_k|^2} = \EE\pra[\big]{\pra{M}_k}\leq C h \, \delta t \sum_{n=0}^{k-1} u^{n}_\infty .
\end{equation}

In the case $r\ge 2$ we will use the second of the aforementioned martingale estimates,  namely Doob's inequality \cite[Proposition 6.16]{Kallenberg2002}: It holds for all $r>1$
\[
  \EE\pra[\big]{ \bra{M_k^*}^r }^\frac{1}{r} \leq \frac{r}{r-1} \EE\pra[\big]{ |M_k|^r}^\frac{1}{r}.
\]
For the remaining statement of the lemma, it is thus enough to estimate the $L^r$ norm of the martingale $M_k$ for all $r>2$. We will furthermore restrict our study to even values of $r$, that is, we assume that $r=2s$ for some $s\in\N$. The general case then follows by Jensen's inequality as for $r\in [1,2)$. We argue by induction over $s$. The induction base is settled for the case $s=1$, which we just have proven in~\eqref{e:est:MnL2}. Hence, we assume the induction hypothesis
\begin{equation}\label{e:est:MnIH}
     \EE\pra[\big]{\abs{M_k}^{r}} \leq C_r \bra*{ h \bra[\bigg]{\delta t \sum_{n=0}^{k-1} u^{n}_\infty  + h }}^{\frac{r}{2}}
\end{equation}
to hold for any $r\in \R$ with  $2 \leq r \leq 2 s-2$.
To begin the induction step, we calculate
\begin{align*}
  \abs{ M_{k+1}}^{2s} =& \abs{M_k + \xi^k}^{2s} = \bra*{ \abs{M_k}^2 + 2\skp{M_k , \xi^k} + \abs{\xi^k}^2}^s \\
  = &\abs{M_k}^{2s} +s\underbrace{|M_k|^{2(s-1)} \bra*{2\skp{M_k , \xi^k} + \abs{\xi^k}^2}}_{=:I}\\
  & + \sum_{\ell=2}^s \binom{s}{\ell} \underbrace{|M_k|^{2(s-\ell)} \bra*{  2\skp{M_k , \xi^k} + \abs{\xi^k}^2}^{\ell}}_{=:\textit{II}_\ell}.
\end{align*}
In the estimation of the term $I$, it is crucial that $|M_k|^{2(s-1)}M_k$ and $\xi^k$ are orthogonal in expectation. Indeed, thanks to \eqref{e:est:xi} and the law of total expectation, it holds that $\EE\pra*{|M_k|^{2(s-1)} \skp{M_k,\xi^k}} = \EE\pra*{|M_k|^{2(s-1)} \, \EE\pra{\skp{M_k,\xi^k}\mid\cF^k}}=0$. Similarly, with the help of~\eqref{e:est2:xi}, we derive $\EE\pra{|M_k|^{2(s-1)}|\xi^k|^2}\lesssim h \, \delta t\,  u_{\infty}^k \EE\pra{|M_k|^{2(s-1)}}$. We have thus shown that
\[
  \EE[I] \leq C h \, \delta t \, u^k_\infty \EE\pra*{ |M_k|^{2(s-1)} }  .
\]
Using essentially the same arguments as in the last estimate, we can get control over the terms $\text{\it II}_{\ell}$. Indeed,  with the help of~\eqref{e:est2:xi}, since $\ell\geq 2$, we have
\begin{align*}
  \EE[\textit{II}_{\ell}] &\leq C \EE\pra[\Big]{ |M_k|^{2s-{\ell}}  \abs{\xi^{k}}^{\ell} } + \EE\pra[\Big]{ |M_k|^{2(s-{\ell})}  \abs{\xi^{k}}^{2\ell} }\\
  &\leq C  \delta t \,  u^{k}_\infty h^{\ell-1}  \EE\pra*{ |M_k|^{2s-\ell} }   +\delta t \, u^{k}_\infty h^{2\ell-1} \EE\pra*{|M_k|^{2(s-\ell)}}   .
\end{align*}
It will be convenient to define $U^{k}_\infty =   \sum_{n=0}^{k-1} u_\infty^{n}$.
Combining these two estimates 
and applying the induction hypothesis~\eqref{e:est:MnIH}, we then obtain
\begin{align*}
  \EE\pra*{\abs{ M_{k+1}}^{2s}} &= \EE\pra[\bigg]{\abs{ M_{k}}^{2s}+ s I + \sum_{\ell=2}^s \binom{s}{\ell} \textit{II}_\ell } \\
  &\leq \EE\pra*{\abs{ M_{k}}^{2s}} + C\, \delta t\, u^k_\infty \sum_{\ell=2}^{2s} \bra*{h \bra{\delta t\, U^{k}_\infty + h } }^\frac{2s-\ell}{2} h^{\ell-1} \\
  &\leq \EE\pra*{\abs{ M_{k}}^{2s}} + C\, \delta t\, u^k_\infty \, h^s \bra[\bigg]{\sum_{\ell=2}^{2s}  \bra{\delta t \, U^{k}_\infty}^{\frac{2s-\ell}{2}} h^{\frac{\ell -2}{2}} + h^{s - 1} } \\
  &\leq \EE\pra*{\abs{ M_{k}}^{2s}} + C\, \delta t\, u^k_\infty \, h^s \bra*{ \bra{\delta t \, U^{k}_\infty}^{s-1} + h^{s-1} },
\end{align*}
where we have used the fact that $\sum_{\ell=2}^{2s}  a^{\frac{2s-\ell}{2}}b^{\frac{\ell-2}{2}} = \sum_{\ell=0}^{2s-2} \sqrt{a}^{2s-2-\ell} \sqrt{b}^{\ell} \leq C (a^{s-1}+b^{s-1})$ in the last inequality. Iterating this estimate and using $M_0=0$, we obtain for any $k\in\N$
that
\[
\EE\pra*{|M_{k}|^{2s}}\le C \sum_{n=0}^{k-1} \delta t\, u_{\infty}^n \, h^s\left(\left(\delta t \, U_{\infty}^n\right)^{s-1}+h^{s-1}\right) \le C h^s\left(\delta t\, U_{\infty}^{k} + h\right)^s.
\]
This proves the statement in  \eqref{e:est:MnIH} for $r=2s$. For intermediate values $r=2s- \sigma$ with $\sigma\in (0,2)$, we estimate by using Jensen's inequality
again
\[
  \EE\pra[\big]{ \abs{M_n}^{r}}^{\frac{1}{r}} = \EE\pra*{ \bra[\big]{\abs{M}_n^{2s}}^\frac{2s-\sigma}{2s}}^{\frac{1}{2s-\sigma}}\leq \EE\pra[\bigg]{ \abs{M}_n^{2s}}^\frac{1}{2s} .
\]
This concludes the proof.
\end{proof}

\section{Proof of the error estimates}\label{S:6}

\subsection{Proof of Theorem \ref{T1}}

We start with the observation that it is enough to consider the case of \emph{nonnegative} (approximate) solutions.
Indeed, in view of the superposition principle  for the continuity equation, that is, $\rho(t) = (\phi_t)_{\#}\rho_0$, it is clear that solutions are nonnegative if the data are. In particular, if the functions $\rho_{\pm}$ denote the solutions corresponding to the initial data $\rho_0^\pm\geq 0$, then $\rho_{\pm}\ge 0$. Moreover, by the uniqueness of the Cauchy problem, it holds that $\rho = \rho_+ -\rho_-$ is the unique solution with initial data $\rho_0 = \rho_0^+ - \rho_0^-$. With regard to Lemma \ref{lem:prop_scheme}, the same holds true in the discrete setting: The solution to the upwind scheme can be split into $\rho_{h} = (\rho_{h})_+ - (\rho_{h})_-$ where $(\rho_{h})_{\pm}$ is the nonnegative discrete solution with initial data $(\rho^{\pm}_0)_h$ (first decomposed then discretized). Now, by the transshipment property \eqref{20} and the triangle inequality \eqref{19} of the Kantorovich--Rubinstein distance, we estimate
\[
\D_r(\rho,\rho_{h}) \le \D_r(\rho_+,(\rho_{h})_+) + \D_r(\rho_-,(\rho_{h})_-).
\]
To prove Theorem \ref{T1}, it is therefore enough to control the distances between the nonnegative densities on the right.

We keep $t\in[t^n,t^{n+1})$ fixed, and split the Kantorovich--Rubinstein distance between $\rho(t) = (\phi_t)_{\#}\rho_0$ and $\rho_h(t) = \psi^n_{\#}\rho_h^0$ according to
\begin{align}
\MoveEqLeft{\D_r(\rho(t),\rho_h(t))} \notag\\
&\le \D_r((\phi_t)_{\#}\rho_0, (\phi_{t^n})_{\#}\rho_0)  + \D_r((\phi_{t^n})_{\#}\rho_0, (\phi_{t^n})_{\#}\rho_h^0)  + \D_r((\phi_{t^n})_{\#}\rho_h^0, \psi^n_{\#}\rho_h^0). \label{17}
\end{align}
Here we have used the triangle inequality for $\D_r$, cf.\ \eqref{19}. The first term in \eqref{17} measures the error caused by the {\em discretization in time}, the second one measures the error due to the {\em discretization of the initial datum} and the third term quantifies the {\em error of the upwind scheme}. The estimates of the first two terms are contained in the following two lemmas.
\begin{lemma}[Error due to time discretization]\label{L1}
There exists a constant $C$ such that
\[
\D_r((\phi_t)_{\#}\rho_0, (\phi_{t^n})_{\#}\rho_0) \le \log\left(\frac{C h}{r}+1\right) \|\rho_0\|_{L^1} .
\]
\end{lemma}
\begin{lemma}[Error due to discretization of initial data]\label{L2}
There exists a constant $C$ such that
\[
  \D_r((\phi_t)_{\#}\rho_0, (\phi_{t})_{\#}\rho_h^0) \le C\, \log\left(\frac{h}{r}+1\right) \bra*{ \| \rho_0\|_{L^1} +  \Lambda^\frac1p \, \|\rho_0\|_{L^q}\, \| u\|_{L^1(W^{1,p})}} .
\]
\end{lemma}

Here and in the following, $\|u\|_{L^1(W^{1,p})}$ can be replaced by the homogeneous part $\|\grad u\|_{L^1(L^p)}$ in situations where $\Omega$ is a convex domain.

To estimate the third term in \eqref{17}, we first estimate the transportation distance with the help of the standard coupling \eqref{e:standard_coupling},
\[
\D_r((\phi_{t^n})_{\#}\rho_h^0, \psi^n_{\#}\rho_h^0) \le \EE_{\rho^0_h}\pra*{\log\left(\frac{|\phi_{t^n}(x) - \psi^n|}r +1\right)} .
\]
Using the evolution laws in \eqref{e:flow:sol} and \eqref{e:Psi:difference} together with the concavity of the logarithm, we obtain
\begin{align*}
\D_r((\phi_{t^n})_{\#}\rho_h^0, \psi^n_{\#}\rho_h^0)
&\le \EE_{\rho^0_h}\pra*{\log\bra[\bigg]{\frac{|\phi_{t^{n-1}}(x) - \psi^{n-1} - \xi^{n-1}|}r +1}} \\
&\quad + \EE_{\rho^0_h}\pra*{\frac{\abs[\big]{\int_{t^{n-1}}^{t^n} u(s,\phi_s(x))\, ds - \delta t \, u_h^{n-1}(\psi^{n-1})}}{|\phi_{t^{n-1}}(x) - \psi^{n-1} - \xi^{n-1}|+r}} .
\end{align*}
This procedure can be repeated. After $n-1$ iterations, we have the estimate
\begin{align}
\D_r((\phi_{t^n})_{\#}\rho_h^0, \psi^n_{\#}\rho_h^0)
&\le \EE_{\rho^0_h}\pra*{\log\bra[\bigg]{\frac1r \abs[\bigg]{\sum_{\ell=0}^{n-1} \xi^{\ell}} + 1}} \label{21}\\
&\quad + \sum_{\ell=0}^{n-1} \EE_{\rho^0_h}\pra*{\frac{\abs[\big]{\int_{t^{\ell}}^{t^{\ell+1}} u(s,\phi_s(x))\, ds - \delta t \, u_h^{\ell}(\psi^{\ell})}}{|\phi_{t^{\ell}}(x) - \psi^{\ell} - \sum_{k=\ell}^{n-1}\xi^{k}|+r}} .\notag
\end{align}
Let us denote the first term by $T_0$ and note that the second term on the right-hand side of the previous estimate is furthermore controlled by the sum $T_1 + T_2 + T_3$, where
\begin{align}
T_1 &:=\sum_{\ell=0}^{n-1}\frac1r\int_{t^{\ell}}^{t^{\ell+1}} \int |u(s,\phi_s(x)) - u(s,\phi_{t^{\ell}}(x))|\,\rho^0_h(dx)\, ds\label{22}\\
T_2 &:= \sum_{\ell=0}^{n-1} \int_{t^{\ell}}^{t^{\ell+1}} \EE_{\rho^0_h}\pra*{\frac{\abs*{  u(s,\phi_{t^{\ell}}(x)) - u(s,\psi^{\ell})}}{|\phi_{t^{\ell}}(x) - \psi^{\ell} - \sum_{k=\ell}^{n-1}\xi^k| +r}}\, ds  \label{23}\\
T_3 &:= \sum_{\ell=0}^{n-1} \frac1r \EE_{\rho^0_h}\pra*{\abs[\bigg]{\int_{t^{\ell}}^{t^{\ell+1}} u(s,\psi^{\ell})\, ds -\delta t \, u_h^{\ell}(\psi^{\ell})}}.\label{24}
\end{align}
We thus have to estimate the terms in \eqref{21}--\eqref{24}. This is the content of the following lemmas.

The term in~\eqref{21} is caused by the numerical diffusion introduced by the upwind scheme, which is manifested as a sum of centered random variables. Our proof of \eqref{21} consists of an application of martingale estimate from Lemma~\ref{L8}.
\begin{lemma}[Estimate of $T_0$ \eqref{21}]\label{L4}
There exists a constant $C$
such that
\[
\EE_{\rho^0_h}\pra*{\log\bra[\bigg]{\frac1r \abs[\bigg]{\sum_{\ell=0}^{n-1} \xi^{\ell}} +1}} \le C \frac1r\left(\sqrt{h \norm{u}_{L^1(L^\infty)}} + h\right) \, \|\rho_0\|_{L^1}  .
\]
\end{lemma}

The next term~\eqref{22} involves a time-shift and we use again the CFL condition~\eqref{CFL_true}. In addition, in this estimate, we use the maximal function of the gradient of~$u$ to bound the difference.
\begin{lemma}[Estimate of $T_1$ \eqref{22}]\label{L5}
There exists a constant $C $ such that
\[
\sum_{\ell=0}^{n-1}\frac1r\int_{t^{\ell}}^{t^{\ell+1}}\int |u(s,\phi_s(x)) - u(s,\phi_{t^{\ell}}(x))|\,\rho^0_h(dx) \, ds \le C\, \frac{h}r \, \Lambda^\frac1p \, \|\rho_0\|_{L^q} \, \| u\|_{L^1(W^{1,p})} .
\]
\end{lemma}
To estimate the term in~\eqref{23}, we use a combination of the maximal function estimate and the martingale estimate from Lemma~\ref{L8}.
\begin{lemma}[Estimate of $T_2$ \eqref{23}]\label{L6}
There exists a constant $C $ such that
\begin{align*}
\MoveEqLeft{\sum_{\ell=0}^{n-1} \int_{t^{\ell}}^{t^{\ell+1}} \EE_{\rho^0_h}\pra*{\frac{\abs*{  u(s,\phi_{t^{\ell}}(x)) - u(s,\psi^{\ell})}}{|\phi_{t^{\ell}}(x) - \psi^{\ell} - \sum_{k=\ell}^{n-1}\xi^k| +r}}\, ds} \\
&\le C\left(1+\frac1r\left(\sqrt{h \norm{u}_{L^1(L^\infty)}} + h\right)\right) \Lambda^\frac1p \, \|\rho_0\|_{L^q} \| u\|_{L^1(W^{1,p})}.
\end{align*}
\end{lemma}
The control of \eqref{24} crucially relies on the particular form of the averaged velocity field $u_h^l$ introduced in \eqref{e:def:netflow} and the regularity of the mesh. To be more specific, our argument is based on the identity \eqref{e:netflow:cubes} which seems to be valid \label{mesh_discussion} on Cartesian tessellations only. At this stage, it is not clear to us how to estimate $T_3$ \eqref{24} in the case of more general, possibly unstructured, tessellations, though some ideas from the construction in \cite{DelarueLagoutiere11} may be relevant. We plan to address this question in future research.

\begin{lemma}[Estimate of $T_3$ \eqref{24}]\label{L7}
There exists a constant $C$ such that
\begin{equation*}
  \sum_{\ell=0}^{n-1} \frac1r \ \EE_{\rho^0_h}\pra*{\abs[\bigg]{\int_{t^{\ell}}^{t^{\ell+1}} u(s,\psi^{\ell})\, ds - \delta t \, u_h^{\ell}(\psi^{\ell})}} \leq C \, \frac{h}{r} \, \Lambda^\frac1p \, \norm{\rho_0}_{L^q} \, \norm{ u}_{L^1(W^{1,p})} .
\end{equation*}
\end{lemma}
A combination of the Lemmas~\ref{L1}--\ref{L7} completes the proof of Theorem~\ref{T1}.

\subsection{Proof of Lemmas~\ref{L1}--\ref{L7}}

In this subsection, we turn to the proofs of Lemmas \ref{L1}--\ref{L7}.
\begin{proof}[Proof of Lemma \ref{L1}]
From the CFL condition \eqref{CFL_true}, we deduce
\[
|\phi_t(x) - \phi_{t^n}(x)| \le \int_{t^n}^t |u(s,\phi_s(x))|\, ds \le \delta t \, u_{\infty}^n \leq C h
\]
for a.e.\ $x\in\Omega$, with the consequence that
\begin{align*}
\D_r((\phi_t)_{\#}\rho_0, (\phi_{t^n})_{\#}\rho_0) &\overset{\mathclap{\eqref{e:standard_coupling}}}{\le} \int \log\left(\frac{|\phi_t(x) - \phi_{t^n}(x)|}{r}+1\right) \rho_0(x)\, dx\\
&\le\log\left(\frac{C h}r+1\right) \|\rho_0\|_{L^1}.
\end{align*}
This proves Lemma \ref{L1}.
\end{proof}

The statement in Lemma \ref{L2} is a stability estimate for the continuity equation which has been recently proved in \cite[Proposition 1]{Seis16a} and builds up on \cite[Proposition 2.2]{BOS}.  In order to have a self-contained representation, we will sketch its short proof for the convenience of the reader.

We need some preparations. At the heart of the proof is a Crippa--De~Lellis-type argument, cf.\ \cite{CrippaDeLellis08}, that allows to estimate integrals of difference quotients by $L^p$ norms of gradients. The argument makes use of the theory of the maximal function operator~$M$, defined for a function $f$ on $ \R^d$ by
\[
(M f)(x) = \sup_{r>0} \avint_{B_r(x)} |f(y)|\, dy.
\]
We will make use of two properties. First, maximal functions bound difference quotients in the sense that 
\begin{equation}
\label{5}
\frac{|f(x) - f(y)|}{|x-y|}\leq C \bra[\big]{ (M\grad f)(x) + (M\grad f)(y) }
\end{equation}
for a.e.\ $x,y$. Furthermore, $M$ maps $L^p$ to $L^p$ for any $p\in (1,\infty]$ with the estimate
\begin{equation}
\label{6}
\|M f\|_{L^p}\leq C \|f\|_{L^p} .
\end{equation}
The first estimate is elementary and can be proved similarly to Morrey's inequality. In fact, its proof is contained in \cite[p.\ 143, Theorem 3]{EvansGariepy92}. The second one can be found in many standard references on harmonic analysis, see, e.g., \cite[p.\ 5, Theorem 1]{Stein70}.

In order to make use of the maximal function estimates, it will be convenient to introduce a Sobolev extension of $u$ to $\R^d$. We thus let $\bar u:\R^d\to\R^d$ denote a Sobolev function with $\bar u = u$ in $\Omega$ and such that
\begin{equation}\label{extension}
\|\bar u\|_{W^{1,p}} \le C \|u\|_{W^{1,p}}.
\end{equation}
The construction of $\bar u$ can be found, for instance, in \cite[p.\ 135, Theorem 1]{EvansGariepy92}.

\begin{proof}[Proof of Lemma \ref{L2}]
  We apply the estimate~\eqref{e:Dist:deriv} with $\rho_1 = (\phi_t)_{\#}\rho_0$ and $\rho_2= (\phi_t)_{\#}\rho_h^0 $ and obtain with the help of \eqref{5} and the marginal conditions \eqref{18}:
\[
\left|\frac{d}{dt} \D_r((\phi_t)_{\#}\rho_0,(\phi_t)_{\#}\rho_h^0)\right| \leq C \io (M\grad \bar u) \,  |(\phi_t)_{\#}(\rho_0-\rho_h^0)|\, dx.
\]
Integration and H\"older's inequality yield
\[
\D_r((\phi_t)_{\#}\rho_0,(\phi_t)_{\#}\rho_h^0) \le \D_r(\rho_0,\rho_h^0) + C \|M\grad \bar u\|_{L^1(L^p)} \|(\phi_t)_{\#}(\rho_0-\rho_h^0)\|_{L^{\infty}(L^q)}.
\]
To bound the second term we invoke  \eqref{6},  \eqref{e:Phi:Lq}, and \eqref{extension} and obtain
\[
\|M\grad \bar u\|_{L^1(L^p)} \|(\phi_t)_{\#}(\rho_0-\rho_h^0)\|_{L^{\infty}(L^q)}\leq C \|\grad  u\|_{L^1(L^p)} \Lambda^\frac1p \|\rho_0- \rho_0^h\|_{L^q} .
\]
For the first term, we recall from the definition of $\rho_h^0$ in \eqref{e:disc_initial} that $\rho_0$ and $\rho_h^0$ share both the same mass on each cell $K\in\cT$. We may thus choose $\pi_K\in \Pi(\chi_K \rho_0,\chi_K\rho_h^0)$ and define $\pi = \sum_{K\in\cT} \pi_K$. By construction, $\pi$ is a transfer plan in $\Pi(\rho_0,\rho_h^0)$. In particular,
\[
\D_{r}(\rho_0,\rho_h^0) \le \sum_{K\in\cT} \iint \log\left(\frac{|x-y|}{r} +1\right)\, d\pi_K(x,y) \le \log\left(\frac{h}r+1\right) \|\rho_0-\rho_h^0\|_{L^1}.
\]
Since $\|\rho_h^0\|_{L^1} = \|\rho_0\|_{L^1} $, the statement follows with the triangle  inequality.
\end{proof}

\begin{proof}[Proof of Lemma \ref{L4}]
We use the sublinearity of the logarithm
\begin{align*}
  \EE_{\rho_h^0}\pra*{\log\bra[\bigg]{\frac1r\abs[\bigg]{\sum_{\ell=0}^{n-1}\xi^{\ell}}+1}} \le  \frac1r \EE_{\rho_h^0}\pra*{\abs[\bigg]{\sum_{\ell=0}^{n-1}\xi^{\ell}}}  ,
\end{align*}
apply the martingale estimate \eqref{e:est:quad:xi} from Lemma \ref{L8}, and recall  the identity $\|\rho_h^0\|_{L^1} = \|\rho_0\|_{L^1}$.
\end{proof}

\begin{proof}[Proof of Lemma \ref{L5}]
We recall that the CFL condition in \eqref{CFL_true} guarantees that $|\phi_s(x) - \phi_{t^{\ell}}(x)|\le C h$ for a.e.\ $x\in\Omega$ and every $s\in[t^{\ell},t^{\ell+1})$, cf.\ proof of Lemma \ref{L1} above. It thus follows via \eqref{5} that
\begin{align*}
\lefteqn{\int_{t^{\ell}}^{t^{\ell+1}} \int |u(s,\phi_s) - u(s,\phi_{t^{\ell}})| \rho_h^0\, dx\,ds }\\
&\leq C
h \int_{t^{\ell}}^{t^{\ell+1}} \int (M\grad \bar u)(s,\phi_s) \rho_h^0\, dx\,ds + C h \int_{t^{\ell}}^{t^{\ell+1}} \int (M\grad\bar  u)(s,\phi_{t^{\ell}})\rho_h^0\, dx\,ds.
\end{align*}
Summing over $\ell$ and applying H\"older's inequality and the bound on the Jacobian \eqref{14} thus yield
\[
\sum_{\ell=0}^{n-1}\int_{t^{\ell}}^{t^{\ell+1}} \int |u(s,\phi_s) - u(s,\phi_{t^{\ell}})| \rho_h^0\, dx\,ds
\leq C h  \Lambda^\frac1p \|M\grad \bar u\|_{L^1(L^p)}\|\rho_h^0\|_{L^q}.
\]
It remains to invoke the fundamental inequality for maximal functions~\eqref{6} and the continuity of the extension operator \eqref{extension} to deduce the statement of the lemma.
\end{proof}

\begin{proof}[Proof of Lemma \ref{L6}]
We first apply H\"older's inequality to the expectation,
\begin{align*}
\lefteqn{\EE_{\rho^0_h}\pra*{\frac{|u(s,\phi_{t^{\ell}}) - u(s,\psi^{\ell})| }{|\phi_{t^{\ell}} - \psi^{\ell} - \sum_{k=\ell}^{n-1}\xi^k| + r}} }\\
&\le \EE_{1}\pra[\Bigg]{\frac{|u(s,\phi_{t^{\ell}}) - u(s,\psi^{\ell})|^p}{|\phi_{t^{\ell}} - \psi^{\ell}|^p }}^{1/p} \EE_{(\rho^0_h)^q}\pra[\Bigg]{\frac{|\phi_{t^{\ell}} - \psi^{\ell}|^{q}}{\left(|\phi_{t^{\ell}} - \psi^{\ell} - \sum_{k=\ell}^{n-1}\xi^k| + r\right)^{q}}}^{1/{q}} .
\end{align*}
The second term on the right-hand side can be bounded with the help of the triangle inequality and the martingale estimate \eqref{e:est:quad:xi} by
\begin{align*}
\lefteqn{\EE_{(\rho^0_h)^q}\pra*{\frac{|\phi_{t^{\ell}} - \psi^{\ell}|^{q}}{\left(|\phi_{t^{\ell}} - \psi^{\ell} - \sum_{k=\ell}^{n-1}\xi^k| + r\right)^{q}}}^{1/q}}\\
&\leq \EE_{(\rho^0_h)^q}\pra{1}^\frac1q + \frac1{r}\EE_{(\rho^0_h)^q}\pra*{\abs[\bigg]{\sum_{k=\ell}^{n-1}\xi^k}^{q}}^{1/q} \leq C \bra*{ 1 + \frac1r \bra*{\sqrt{h \norm{u}_{L^1(L^\infty)}} + h}} \norm{\rho^0_h}_{L^q} .
\end{align*}
Notice that the martingale estimate extends to sums starting at $\ell$ via the triangle inequality.
For the first term, we have
\[
\EE_{1}\pra*{\frac{|u(s,\phi_{t^{\ell}}) - u(s,\psi^{\ell})|^p}{|\phi_{t^{\ell}} - \psi^{\ell}|^p }} \leq C \int (M\grad\bar u)(s,\phi_{t^{\ell}}(x))^p \, dx + C\, \EE_{1}\pra[\big]{(M\grad \bar u)(s,\psi^{\ell})^p}
\]
as a consequence of \eqref{5}. Thus, with regard to \eqref{e:Phi:Lq}, \eqref{e:scheme:Linfty}, \eqref{6}, and \eqref{extension}, the latter yields
\[
\EE_{1}\pra*{\frac{|u(s,\phi_{t^{\ell}}) - u(s,\psi^{\ell})|^p}{|\phi_{t^{\ell}} - \psi^{\ell}|^p }}^\frac1p \leq C  \, \Lambda^\frac1p \, \| u(s,\tacka)\|_{W^{1,p}} .
\]
Combining the previous estimates, integration over $[t^\ell,t^{\ell+1}]$ and doing the summation in $\ell$ yields the result.
\end{proof}

\begin{proof}[Proof of Lemma \ref{L7}]
By the assumption of a Cartesian tessellation, we have for each control volume $K\in \cT$ the identity
\[
 \sum_{L\sim K}\nu_{KL}(b\cdot \nu_{KL})^+ = b
\]
for any vector $b\in \R^d$. In particular, choosing $b = \avint_{t^{\ell}}^{t^{\ell+1}} \avint_K u\, dx\,ds$, where $K\in\cT$ is such that $\psi^{\ell}\in K$, it holds that
\begin{align*}
\MoveEqLeft \int_{t^{\ell}}^{t^{\ell+1}} u(s,\psi^{\ell})\, ds - \delta t \, u_h^{\ell}(\psi^{\ell})\\
 = & \int_{t^{\ell}}^{t^{\ell+1}}\avint_K \left(u(s,\psi^{\ell}) - u(s,x)\right)\, dx\,ds\\
& + \sum_{L\sim K} \nu_{KL} \left( \bra[\bigg]{\int_{t^{\ell}}^{t^{\ell+1}} \avint_K u\cdot\nu_{KL}\, dx\,ds }^{\!+} - \bra[\bigg]{\int_{t^{\ell}}^{t^{\ell+1}} \avint_{K\edge L}u\cdot \nu_{KL}\, d\Ha^{d-1}ds }^{\!+}\right)\\
=: & I + \textit{II} .
\end{align*}
In view of \eqref{5}, the first term is controlled as follows:
\begin{align*}
\abs{I} &\leq C h\left( \int_{t^{\ell}}^{t^{\ell+1}} (M\grad\bar  u)(s,\psi^{\ell})\, ds + \int_{t^{\ell}}^{t^{\ell+1}}\avint_K  (M\grad\bar u)(s,x)\,dx\,ds\right).
\end{align*}
For the second one, we use the fact that $(\cdot)^+$ is $1$-Lipschitz continuous and compute
\begin{align*}
\abs{\textit{II}\,} &\le  \sum_{L\sim K} \left| \int_{t^{\ell}}^{t^{\ell+1}} \avint_K u\cdot\nu_{KL}\,dx\,ds - \int_{t^{\ell}}^{t^{\ell+1}} \avint_{K\edge L}u\cdot \nu_{KL}\, d\Ha^{d-1}\,ds\right|\\
&\le C \int_{t^{\ell}}^{t^{\ell+1}} \avint_{\partial K} \abs[\bigg]{ u -\avint_K u \, dx}  \, d\Ha^{d-1}\,ds.
\end{align*}
Now, we use the estimate
\begin{equation*}
\int_{\partial K} \left| u - \avint_K u\, dy\right|\, d\Ha^{d-1} \leq C \int_K |\grad u|\, dx ,
\end{equation*}
which is a consequence of the trace estimate
\[
\int_{\partial K} |v|\, d\Ha^{d-1} \leq C\bra*{\frac{1}{h} \int_K |v| \, dx  + \int_K |\grad v|\, dx}
\]
(cf.\ \eqref{e:trace}) applied to $v = u- \avint_K v\, dx$, and of the standard Poincar\'e estimate
\[
\int_K \left| u - \avint_K u\, dy\right|dx \leq h \int_K |\grad u|\, dx.
\]
Therewith, we obtain the estimate
\begin{align*}
  \abs{\textit{II}\,} \leq \frac{C}{\abs{\partial K}}  \int_{t^{\ell}}^{t^{\ell+1}} \int_K |\grad u|\, dx\, ds \leq C h \int_{t^{\ell}}^{t^{\ell+1}} \avint_K |\grad u|\, dx\, ds .
\end{align*}
because $|K|/|\partial K| \lesssim |K|^{1/d} \lesssim h$ thanks to the isoperimetric inequality. Applying the expectation and doing the push-forward then yields
\[
\EE_{\rho_h^0}\pra*{\abs[\bigg]{\int_{t^{\ell}}^{t^{\ell+1}} u(s,\psi^{\ell})\, ds - \delta t \, u_h^{\ell}(\psi^{\ell})}} \leq C h \int_{t^{\ell}}^{t^{\ell+1}}\int \bra[\big]{ \abs{M\grad u} + \abs{\grad u}}  \, \psi^l_{\#}\rho_h^0(dx) \, ds,
\]
where we used the identity
\[
\EE_{\rho_h^0}\pra*{\avint_{K(\psi^{\ell})} f\, dx} = \EE_{\rho_h^0}\pra*{\int_{J^{\ell}} f\, dx} = \sum_{K\in\cT} \rho_K^{\ell} \int_K f\, dx,
\]
in which $K(x)$ denotes the control volume in $\cT$ that contains $x$. Hereby, in the first expectation, $\rho_h^0$ is interpreted as a function from $\Omega$ to $\R$ and hence the first expectation is an integral, whereas in the second one it is considered as a vector in $\R^{\cT}$ and the second expectation is thus a sum. We use H\"older's inequality, the fundamental estimate for maximal functions in \eqref{6}, the continuity of the extension operator \eqref{extension} and estimate~\eqref{e:scheme:Linfty} to conclude.
\end{proof}

\section{Optimality} \label{Ex1}
Our intention in this section is to demonstrate that our main result is (almost) optimal with regard to two aspects:

\begin{enumerate}
\item We state a simple example which illustrates that within the setting of this paper  one cannot expect to prove a priori upper bounds on (polynomial) convergence rates in \emph{strong} Lebesgue norms.  To be more specific, for any small~$\eps$ we find an initial configuration such that the approximate solution given by the upwind scheme converges towards the exact solution of the continuity equation with a  rate not faster than order $\eps$, see \eqref{opt1} below. Taking the limit $\eps\searrow 0$, this entails that uniform rates cannot exist for strong norms. It is thus natural to seek for estimates on the rate of \emph{weak} convergence, as provided in our Theorem \ref{T1}.

Notice that this observation is not a contradiction with the error analysis conducted, for instance, in \cite{Merlet07,MerletVovelle07}: In these works, the authors study convergence rates under regularity assumptions on the initial datum: They assume that $u_0$ has $BV$ regularity. Our theory, however, is valid for data that are merely integrable.

\item  Our computations show that our findings in Theorem \ref{T1} are \emph{almost} optimal in the following sense: For any small $\eps$, there exist initial configurations for which we can bound the order of weak convergence from below by $1/2-\eps$, see \eqref{opt2} below. This lower bound almost matches the $1/2$-a priori upper bound from Theorem~\ref{T1}.

Apart the unpleasant  fact that upper and lower bounds do not exactly agree, there is a second mismatch with regard to the measures of weak convergence. We are not able to bound the logarithmic Kantorovich--Rubinstein distance~$\D_r$ suitably from below. Instead, we study the slightly larger Kantoro\-vich--Rubinstein  distance with Euclidean cost
\[
W_1(\rho,\rho_h) = \inf_{\pi\in \Pi(\rho,\rho_h)} \iint |x-y|\, d\pi(x,y).
\]
This distance is frequently referred to as $1$-Wasserstein distance. By the Kantorovich--Rubinstein theorem \cite[Theorem 1.14]{Villani03}, it satisfies the duality formula
\begin{equation}
\label{W1_duality}
W_1(\rho,\rho_h) = \sup_{\psi}\left\{\int \psi(\rho-\rho_h)\, dx:\: |\psi(x)-\psi(y)|\le |x-y|\right\}.
\end{equation}
\end{enumerate}

Let us now consider the advection equation with a constant velocity field $u\equiv U>0$ on $\R$,
\begin{align*}
  \partial_t \rho + U \partial_x \rho = 0 \qquad\text{and}\quad \rho(0,x) = \rho_0(x).
\end{align*}
Its exact solution is given by $\rho(t,x) = \rho_0(x-tU)$. To define the corresponding approximate solution, let us choose the control volumes $K=h [k,k+1)\subset \R$ for some small $h>0$, and we write $\rho_k^n$ instead of $\rho_K^n$ for solutions of the upwind scheme~\eqref{e:upwind}. Notice that the latter reduces to
\begin{equation*}
  \rho_k^{n+1} = \rho_k^n - \frac{\delta t\, U}{h} \rho_k^n + \frac{\delta t\, U}{h} \rho_{k-1}^n = \bra*{1- \frac{\delta t\, U}{h}} \rho_k^n + \frac{\delta t\, U}{h} \rho_{k-1}^n .
\end{equation*}
Next, we choose the time step $\delta t$ size such that $\delta t \, U=   h/2$, which in particular satisfies the CFL condition~\eqref{e:CFL_num}. Moreover,  the scheme becomes in this simple case $\rho_k^{n+1} = \frac{1}{2}\bra*{ \rho_k^n + \rho_{k-1}^n}$. By, iterating this identity, we arrive at
\begin{equation}\label{e:adv:solution}
  \rho_k^n = \frac{1}{2^n} \sum_{m=0}^n \binom{n}{m} \rho_{k-m}^0 .
\end{equation}

In order to prove the aforementioned optimality of our error estimate in Theorem~\ref{T1}, we have to choose sufficiently rough data. For some parameter $s\in [0,1)$, we choose the following (singular) Riemann problem like initial distribution
\begin{equation*}
  \rho_0(x) = \begin{cases}
                0 &, \text{ for } x \leq 0 \\
                x^{-s} &, \text{ for } x \in (0,1] \\
                0 &, \text{ for } x > 1 .
              \end{cases}
\end{equation*}
By the explicit solution to the continuous problem, we have for any $t>0$ and all $x\in [0,tU]$ that $\rho_h(t,x) - \rho(t,x) = \rho_h(t,x)$. This error is caused by the numerical diffusion and we expect it to be the main contribution to the total error. 

Our argumentation will be based on duality. We thus let $\psi$ be a suitable nonnegative test function with $\spt \psi \subseteq [0,tU]$. Further properties of $\psi$ will be specified later.

Suppose now that $t\in \left[t^n,t^{n+1}\right)$ for some positive even number $n = 2(\ell+1)$, so that $\ell+1 \le t U/h< \ell + 3/2$. We then have
\[
\int \psi(\rho_h(t) - \rho(t))\, dx = \int_0^{tU} \!\psi\,\rho_h(t)\, dx \ge h\sum_{k=0}^{\ell} \psi_k\rho_k^n \overset{\eqref{e:adv:solution}}{=} h \sum_{k=0}^{\ell}  \sum_{m=0}^n\frac1{2^n} \binom{n}{m} \psi_k\rho_{k-m}^0,
\]
where $\psi_k = \avint_{kh}^{(k+1)h} \psi\, dx$. Notice that $\rho_{k-m}^0=0$ for $k<m$. Hence, changing variables and the order of summation, the latter turns into
\[
\int \psi(\rho_h(t)-\rho(t))\, dx \ge \frac{1}4 \sum_{m=0}^{\ell} \frac1{4^\ell} \binom{2\ell +2}{\ell - m} \underbrace{h \sum_{k=0}^m \psi_{\ell-k}\rho_{m-k}^0}_{=:S(m)}.
\]
Notice that the right-hand side is furthermore decreased if we restrict the summation over $m$ to the set $\llb 0,\floor{\sqrt{\ell}}\rrb$ and if we substitute the displayed binomial coefficient by $\binom{2\ell}{\ell- m}$. Moreover, the de Moivre--Laplace theorem yields
\[
\frac1{4^{\ell}}\binom{2\ell}{\ell-m} = \binom{2\ell}{\ell-m}\left(\frac12\right)^{\ell-m}\left(\frac12\right)^{2\ell - (\ell-m)} \approx \frac1{\sqrt{\pi \ell}} e^{-\frac{m^2}{\ell}}\gtrsim \frac1{\sqrt{\ell}}
\]
for $\ell$ sufficiently large. Here, by $a \approx b$ we understand $a=b(1+o(1))$ as $\ell\gg1$. We thus have
\[
\int \psi(\rho_h(t)-\rho(t))\, dx \gtrsim \frac1{\sqrt{\ell}}  \sum_{m=0}^{\floor{\sqrt{\ell}/2}}  S(m).
\]

We now address the $L^1$ lower bound. For this purpose, we choose $\psi(x) = 1$ for $x\in [0,tU]$, and hence also $\psi_k=1$ for all $k \in \llb0, \ell\rrb$, and we obtain
\[
S(m) = h\sum_{k=0}^m \rho_k^0  = \int_0^{(m+1)h} \rho_0\, dx \ge \int_0^{mh} \frac{dx}{x^s} = \frac{(mh)^{1-s} }{1-s},
\]
if $h$ is small enough so that $mh \le 1$. 
Therewith, we arrive at
\begin{align*}
  \norm{\rho_h(t)-\rho(t)}_{L^1}& = \sup_{\|\psi\|_{L^{\infty}}\le 1} \int \psi(\rho_h(t)-\rho(t))\, d x \gtrsim \frac{h^{1-s}}{1-s} \frac1{\sqrt{\ell}} \sum_{m=0}^{\floor{\sqrt{\ell}/2}}  m^{1-s}\sim \frac{\sqrt{h}^{1-s}}{1-s},
\end{align*}
because $\ell \sim tU/h\sim 1/h$ for $t\ge 1$ and $U\sim 1$. Hence, setting $\eps = (1-s)/2$, this computation shows that
\begin{equation}\label{opt1}
\lim_{h\to 0} h^{-\eps} \|\rho_h-\rho\|_{L^1((0,1)\times (0,R))} \gtrsim 1,
\end{equation}
for any $R$ sufficiently large.

The computation for the Wasserstein distance is similar. We choose the $1$-Lipschitz function $\psi(x)=tU-x$ on $[0,tU]$ and obtain
\[
  S(m) \approx \int_0^{(m+1)h}\left(tU - (\ell-m)h - x\right)\frac{dx}{x^s} \geq \frac{(mh)^{2-s}}{(1-s)(2-s)}.
\]
This leads to the lower bound
\[
  W_1\bra*{\rho_h(t),\rho(t)} 
 \overset{\eqref{W1_duality}}{ \gtrsim} \frac{h^{2-s}}{(1-s)^2(2-s) \sqrt{ \ell}} \sum_{m=0}^{\floor{\sqrt{\ell}/2}}m^{2-s} \gtrsim \frac{\sqrt{h}^{2-s}}{1-s}.
 \]
By choosing $s = 1-\eps$, we thus find the almost optimal lower bound
\begin{equation}\label{opt2}
W_1(\rho_h(t),\rho(t)) \sim \sqrt{h}^{1+\eps}.
\end{equation}

We finally remark that it is not clear to us whether the weak convergence rates are optimal for more regular, e.g.\ $BV$, data. Our numerical experiments in Section~\ref{S:numerics} suggest that this could be the case, cf.\ Figure \ref{fig2}.

\section{Discussion}\label{S:discussion}

Let us finally discuss possible extensions of our main result. It would be desirable to remove the restriction to Cartesian meshes. The major obstacle consists in the incompatibility of the construction of stochastic characteristics with more general meshes. A way to overcome this in the Lipschitz setting was proposed in~\cite{DelarueLagoutiere11}. At this point, it is not clear to us how to adapt this approach under the weaker regularity assumptions of the present work.

Another question concerns the applicability of our approach to the implicit upwind scheme. We are positive that this application is possible. The argumentation, however, rather relies on the Eulerian specification. This is ongoing research.

We remark that in order to establish stability estimates for continuity equations it is not essential that the system is conservative. In fact, in \cite{Seis16a}, arbitrary source terms are included in the right-hand side of \eqref{1}. The restriction to conservative flows in the present paper is however crucial as it allows for a clean probabilistic interpretation of the scheme. In this context it should be mentioned that it is currently unclear how to extend the theory from \cite{Seis16a} to the transport equation in non-divergence form or to nonlinear conservation laws or systems. For the same reason, the present convergence analysis does not directly apply to the associated upwind schemes.

Finally, there is a way to make sense to the continuity equation in the case of measure valued solutions. The underlying well-posedness theory is based on the notion of renormalized solutions which was first introduced in \cite{DiPernaLions89}. Whether the present work extends to this framework is not obvious to the authors.

\bibliography{coarsening}
\bibliographystyle{abbrv}
\end{document}